\documentclass[reqno]{amsproc}
\usepackage{tikz}
\usepackage{amstext,amsmath,amssymb,amsfonts}
\usepackage[latin1]{inputenc}
\usepackage{epsfig}
\usepackage{hyperref}
\usepackage{color}

\usepackage{geometry}
\usepackage{amsthm}

\definecolor{myurlcolor}{rgb}{0,0,0.7}

\newtheorem{rmq}{Remark}[section]
\newtheorem{dfn}{Definition}[section]
\newtheorem{lem}{Lemma}[section]
\newtheorem{thm}{Theorem}[section]

\newcommand{\bprof}{\begin{prof}}
\newcommand{\eprof}{\end{prof}}
\newenvironment{prof}[1][Proof]{\textbf{#1.} }{\ \rule{0.5em}{0.5em}}
\usepackage{cite}

\newcommand{\bea}{\begin{eqnarray}}
\newcommand{\eea}{\end{eqnarray}}
\newcommand{\beq}{\begin{equation}}
\newcommand{\eeq}{\end{equation}}
\newcommand{\enn}{\nonumber \end{equation}}
\newcommand{\beqs}{\begin{eqnarray*}}
\newcommand{\eeqs}{\end{eqnarray*}}

 \newcommand{\cE}{\mathcal{E}}
\newcommand{\cT}{\mathcal{T}}

\def\cN{{\mathcal N}}

\newcommand{\curl}{\mathop{\rm curl}\nolimits}
\newcommand{\dive}{\mathop{\rm div}\nolimits}

\def\bw{\textbf{w}}
\def\bu{\textbf{u}}

\def\bn{\textbf{n}}

\def\Om{\Omega}
\def\bff{\textbf{f}}

\title[A posteriori error estimation for the Navier-Stokes/Darcy coupled problem]
{Residual-based a posteriori error estimates for a conforming finite element discretization of the Navier-Stokes/Darcy 
coupled problem}

\author{Koffi Wilfrid Houedanou$^{(a,b)}$}
\email{a) khouedanou@yahoo.fr}
\address{Universit\'e d'Abomey-Calavi (UAC), Rep. du B\'enin.}
\email{b) houedanou@aims.ac.za}
\address{African Institute for Mathematical Sciences (AIMS) South Africa.}

\author{Jamal Adetola $^{(c)}$}
\email{c)adetolajamal58@yahoo.com }
\address{
Institut de Math\'ematiques et de Sciences Physiques (IMSP), Rep. du B\'enin.
}

\author{Bernardin Ahounou $^{(d)}$}
\email{d) bahounou@yahoo.fr}
\address{
Universit\'e d'Abomey-Calavi (UAC), Rep. du B\'enin.
}

\begin{document}

\maketitle
\begin{Large}
\begin{abstract}\normalsize
We consider in this paper, a new a posteriori residual type error estimator of a conforming mixed
finite element method for the coupling of fluid flow with porous media flow on isotropic meshes. 
Flows are governed by the Navier-Stokes and Darcy equations, respectively, and the corresponding transmission conditions are 
given by mass conservation, balance of normal forces, and the Beavers-Joseph-Saffman law. 
The finite element subspaces consider Bernardi-Raugel and Raviart-Thomas elements for the velocities, piecewise constants
for the pressures, and continuous piecewise linear elements for a Lagrange multiplier defined
on the interface. The posteriori error estimate is based on a suitable evaluation on the residual of the finite element 
solution. It is proven that the a posteriori error estimate provided  in this paper is both reliable and efficient.
In addition, our analysis can be extended to other finite element subspaces yielding a stable Galerkin scheme.\\
{\bf Mathematics Subject Classification [MSC]:} 74S05, 74S10, 74S15,
74S20, 74S25, 74S30.
\end{abstract}

\tableofcontents
\section{Introduction}
There are many serious problems currently facing the world in which the coupling between groundwater and surface water is 
important. These include questions such as predicting how pollution discharges into streams, lakes, and rivers making its way into 
the water supply. This coupling is also important in technological applications involving filtration.
In particular, for specific applications we refer to flow in vuggy porous media appearing in petroleum extraction 
\cite{AB:2007,AL:2006}, 
groundwater system in karst aquifers \cite{FHKH:2009,M:2012}, 
reservoir wellbore\cite{ASKS:2016}, industrial filtrations \cite{HWNW:2006,N:1998}, 
topology optimization \cite{GP:2006}, and blood 
motion in tumors and microvessels \cite{PF:2003, SF:2005}. 
We  refer to the nice overview \cite{27} and the references therein for its physical background, modeling, and standard numerical 
methods. 
One of the most popular models utilized to describe the aforementioned interaction is the Navier-Stokes/Darcy (or 
Stokes-Darcy) model, which consists in a set of differential equations where the Navier-Stokes (or Stokes) problem is 
coupled with the Darcy model through a set of coupling equations acting on a common interface given by masse conservation, 
balance of normal forces, and the so called Beavers-Joseph-Saffman condition. The
Beavers-Joseph-Saffman condition was experimentally derived by Beavers and Joseph in \cite{23}, modified by Saffman in \cite{44}, 
and later mathematically justified in \cite{32,17,48,37}.

A posteriori error estimators are computable quantities, expressed in terms of the discrete solution  and of the data that measure the actual discrete
errors without the knowledge of the exact solution. They are essential to design adaptive mesh 
refinement  algorithms  which equi-distribute the computational effort and optimize the approximation efficiency. 
Since the pioneering work of Babu\v{s}ka and Rheinboldt \cite{babuska:78a},   adaptive finite element methods based on 
a posteriori error estimates have been extensively investigated. 

A posteriori error estimations have been well-established for the  coupled Stokes-Darcy problem on isotropic meshes, mainly for 
 $2D$ domains \cite{43,49,46,47,AHN:15} and 
 recently on anisotropic meshes \cite{HA:2016,H:2015}. However, only few works exist 
 for the coupled Navier-Stokes/Darcy  problem, see for instance \cite{SGR:2016, HAN:2014}. Up to the author's knowledge, 
 the first work dealing with adaptive algorithms for the Navier-Stokes/Darcy coupling is \cite{HAN:2014}, where 
 an a posteriori error estimator for a discontinuous Galerkin approximation of this coupled problem with 
 constant parameters is proposed. In \cite{SGR:2016}, the authors have derived a reliable and efficient residual-based a 
 posteriori error estimator for the three dimensional version of 
 the augmented-mixed method introduced in \cite{CGOS:2015}.
 The finite element subspaces that they have employed are piecewise
constants, Raviart-Thomas elements of lowest order, continuous piecewise linear elements,
and piecewise constants for the strain,
Cauchy stress, velocity, and vorticity in the fluid, respectively, whereas Raviart-Thomas 
elements of lowest order for the velocity,
piecewise constants for the pressure, and continuous piecewise linear elements for the traces,
are considered in the porous medium.
The authors in \cite{MR:2016} consider the standard mixed formulation in the Navier-Stokes domain and the dual-mixed 
 one in the Darcy region, which yields the introduction of the trace of the porous medium pressure as a suitable Lagrange 
 multiplier. The finite  element subspaces defining the discrete formulation employ Bernardi-Raugel and Raviart-Thomas 
 elements for the velocities, piecewise constants for the pressures, and continuous piecewise linear elements for the 
 Lagrange multiplier. An a priori error analysis is performed with some numerical tests confirming the convergence rates.
 
In this work, we develop an a posteriori error analysis for the finite element method studied in \cite{MR:2016}. 
 The a posteriori error estimate is based on a suitable evaluation on the residual of the finite element
 solution. We further prove that our a posteriori error estimator is both reliable and efficient. These main 
 results are summarized in Theorems \ref{uperbound} and \ref{Lowerbound}.
 The difference between our paper and the reference \cite{SGR:2016} is that
 our analysis uses the standard mixed formulation in the Navier-Stokes domain and 
 the dual-mixed one in the Darcy region, and another family of finite elements to approach the solution. 
 In addition, it's independent
 of the finite elements employed to stabilize the scheme in \cite{MR:2016}. 
 Indeed, no interpolation operator for example linked to the finite elements used in this work 
 is exploited in our a posteriori error analysis.
Consequently,  it can be extended to other finite element subspaces yielding a stable Galerkin scheme. 

The  rest of this work is organized  as follows.
Some  preliminaries and  notation are given in  Section  \ref{sec:r1}. In Section \ref{sec:r2}, the a posteriori 
error estimates are derived. The reliability analysis is carred out in Section \ref{sec:up}, whereas in 
Section \ref{sec:ul} we provide the efficiency analysis. Finally we offer our 
conclusion and the further works in Section \ref{sec:r3}.

\section{Preliminaries and notation}\label{sec:r1}
\subsection{Model problem}
For simplicity of exposition we set the problem in $\mathbb{R}^2$. However, our study can be extended to the 
$3D$ case with few modifications \cite{MR:2016,AHN:15}.
We consider the model of a flow in a bounded domain $\Omega\subset \mathbb{R}^2$, consisting of a 
porous medium domain $\Omega_D$, where the flow is a Darcy flow, and an  open region $\Omega_S=\Omega\smallsetminus 
\overline{\Omega}_D,$ where the flow is governed by the Navier-Stokes equations. The two regions are separated by an interface 
$\Sigma=\partial \Omega_D\cap \partial \Omega_S.$ Let $\Gamma_*=\partial \Omega_*\smallsetminus \Sigma$, $*\in\{S,D\}$.
Each interface and boundary is assumed to be polygonal. We denote by $\textbf{n}_S$ (resp. $\textbf{n}_D$) the 
unit outward normal vector along $\partial \Omega_S$  (resp. $\partial \Omega_D$). Note that on the interface 
$\Sigma$, we have $\textbf{n}_S=-\textbf{n}_D$. 
The FIGURE \ref{F1}  gives a schematic representation of the geometry.
\begin{figure}[htpb]
\begin{center}
\tikzstyle{grisEncadre}=[thick, dashed, fill=gray!20]
\begin{tikzpicture}[scale=0.75]
color=gray!100;
 \draw (1,1)--(7,1);
 \draw (1,1)--(1,5.5);
 \draw (1,5.5)--(7,5.5);
 \draw (7,1)--(7,5.5);
 \draw (1,3)--(7,3);
 \draw [grisEncadre](1,1) rectangle (7,3);
 \draw (4,1.8) node [above]{$\mbox{  $\Omega_D$ }$};
 \draw (4,4) node [above]{$\mbox{  $\Omega_S$ }$};
 \draw [>=stealth,->] [line width=1pt](2,3)--(2,3.5) node [right]{$\textbf{n}_D$};
 \draw [>=stealth,->] [line width=1pt](6,3)--(6,2.5) node [right]{$\textbf{n}_S$};
 \draw [>=stealth,->] [line width=1pt](4,2.8)--(4.7,2.8) node [below]{$\tau$};
 \draw (4.5,3.7) node [below]{$\Sigma$};
 \draw[line width=1pt](1,1)--(1,3) node[midway,above,sloped]{$\Gamma_D$};
 \draw[line width=1pt](1,1)--(7,1) node[midway,below,sloped]{$\Gamma_D$};
 \draw[line width=1pt](7,1)--(7,3) node[midway,below,sloped]{$\Gamma_D$};
 
 \draw[line width=1.pt](1,3)--(1,5.5) node[midway,above,sloped]{$\Gamma_S$};
 \draw[line width=1.pt](1,5.5)--(7,5.5) node[midway,above,sloped]{$\Gamma_S$};
 \draw[line width=1.pt](7,3)--(7,5.5) node[midway,below,sloped]{$\Gamma_S$};
 \end{tikzpicture}
\end{center}
\caption{\footnotesize{\textbf{Domains for the $2\mbox{D}$ Navier-Stokes/Darcy model.}}}
\label{F1}
\end{figure}
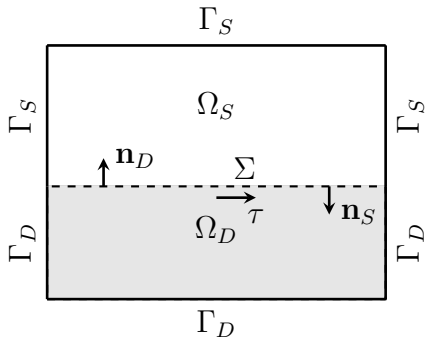

 For any function $v$ defined in 
$\Omega$, since its restriction to  $\Omega_S$ or to $\Omega_D$ could play a different mathematical 
roles (for instance their traces on $\Sigma$), we will set
$v_S=v_{|\Omega_S}$ and $v_D=v_{|\Omega_D}$.

In  $\Omega_*$, $*\in\{S,D\}$ we denote   by $\textbf{u}_*$ the fluid velocity and  by $p_*$ the pressure.
The motion of the  fluid  in $\Omega_S$ is described by the Navier-Stokes equations
\begin{eqnarray}\label{I1}
 \left\{
\begin{array}{ccccccccc}\label{r}
 -2\mu\dive \textbf{e}(\textbf{u}_S)+ \nabla p_S+\rho\left(\textbf{u}_S\cdot \nabla\right)\textbf{u}_S &=
 &\textbf{f}_S &\mbox{ in }&  &\Omega_S,&\\
\dive \textbf{u}_S&=&0  &\mbox{  in   }&  &\Omega_S,&\\
\textbf{u}_S&=&\textbf{0}  &\mbox{  on  }&  &\Gamma_S,&
\end{array}
\right.
\end{eqnarray}
while in the porous medium $\Omega_D$, by Darcy's law
\begin{eqnarray}\label{I2}
 \left\{
\begin{array}{cccccccccccc}\label{rDarcy}
 \textbf{K}^{-1}\mathbf{u}_D+\nabla p_D
 &=&\textbf{f}_D &\mbox{ in }&   &\Omega_D,&\\
\dive \textbf{u}_D&=& 0  &\mbox{ in  }&  &\Omega_D& \\
\textbf{u}_D\cdot\textbf{n} _D& =& 0  &\mbox{  on }&   &\Gamma_D.&
\end{array}
\right.
\end{eqnarray}
Here, $\mu> 0$ is the dynamic viscosity of the  fluid, $\rho$ is its density, $\textbf{f}_S$ is a given external force, 
$\textbf{f}_D$ is a given external force that accounts for gravity, i.e. $\textbf{f}_D=\rho \textbf{g}$ where 
$\textbf{g}$ is the gravity acceleration. $\dive$ is the usual divergence operator and $\textbf{e}$ is the strain rate 
tensor  defined by:
\begin{eqnarray*}
 \textbf{e}(\psi)_{ij}:=\frac{1}{2} \left(\frac{\partial \psi_i}{\partial x_j}+
 \frac{\partial \psi_j}{\partial x_i}\right), \mbox{   } 1\leqslant i,j\leqslant 2,
\end{eqnarray*} 
 and $\textbf{K}\in [L^{\infty}(\Omega_D)]^{2\times 2}$  a symmetric and uniformly positive definite tensor in $\Omega_D$ 
 representing the rock permeability $\kappa$ of the porous medium divided by the dynamic viscosity $\mu$ of the 
 fluid. Throughout the paper we assume that there exits $C> 0$ such that 
 $$ \xi\cdot \textbf{K}\cdot \xi\geq C \|\xi\|_{\mathbb{R}^2}^2,$$ for almost all $x\in \Omega_D$, and for all 
 $\xi\in\mathbb{R}^2$.\\
Finally we consider the following  interface conditions on $\Sigma :$
\begin{eqnarray} \label{cd1}
\textbf{u}_S\cdot \textbf{n}_S+\mathbf{u}_D\cdot \textbf{n}_D&=&0, \\
\label{cd2}
p_S-2\mu \textbf{n}_S\cdot\textbf{e}(\textbf{u}_S)\cdot \textbf{n}_S&=&p_D,\\
\label{cd3}
\frac{\sqrt{\tau\cdot\kappa\cdot\tau}}{\alpha_d\mu} \textbf{n}_S\cdot\textbf{e}(\textbf{u}_S)
\cdot\tau&=&-\textbf{u}_S\cdot\tau,
\end{eqnarray}
where $\alpha_d$  is a 
dimensionless constant which depends only on the geometrical characteristics of 
the porous medium.
Here,  Eq. (\ref{cd1}) represents mass conservation, Eq. (\ref{cd2}) the balance of normal forces, and Eq. (\ref{cd3}) the
Beavers-Joseph-Saffman conditions.  

Eqs. (\ref{I1}) to (\ref{cd3}) consist of the model of the coupled Navier-Stokes and 
Darcy flows problem that we will study below.
\subsection{The variational formulation}
In this section we introduce the weak formulation derived in \cite[Section 2.2]{MR:2016} for the coupled 
problem given by (\ref{I1}) to (\ref{cd3}). To this end, let 
 us first introduce further notations and definitions. In what follows, given $*\in \{S,D\}$, 
 $u,v\in L^2(\Omega_*)$, $\textbf{u}, \textbf{v}\in [L^2(\Omega_*)]^2$, and 
 $\textbf{M}, \textbf{N}\in [L^2(\Omega_*)]^{2\times 2}$, we set 
$$(u,v)_{*}:=\int_{\Omega_*}uv, \mbox{     }\hspace*{1cm} (\textbf{u},\textbf{v})_{*}:=
\int_{\Omega_*}\textbf{u}\cdot\textbf{v}, \mbox{  and  }\hspace*{1cm} (\textbf{M},\textbf{N})_{*}:=\int_{\Omega_*}\textbf{M}:\textbf{N},$$
where, given two arbitrary tensors $\textbf{M}$ and $\textbf{N}$, 
$$\textbf{M}:\textbf{N}:=\mbox{tr}(\textbf{M}^t\textbf{N})=\displaystyle\sum_{i,j=1}^2M_{ij}N_{ij},$$
where the superscript $^t$ denotes transposition.

We use the standard terminology for Lebesgue and Sobolev spaces. In addition, if $\mathcal{O}$ is a domain, 
given and $r\in\mathbb{R}$ and $p\in [1,\infty[,$ we define $\textbf{H}^r(\mathcal{O}):=[H^r(\mathcal{O})]^2$ and 
$\textbf{L}^p(\mathcal{O}):=[L^p(\mathcal{O})]^2$. For $r=0$ we write $\textbf{L}^2(\mathcal{O})$ and 
$L^2(\Gamma)$ instead of $\textbf{H}^0(\mathcal{O})$ and $H^0(\Gamma)$, respectively, where $\Gamma$ is a closed 
Lipschitz curve. The corresponding norms are denoted by $\|\cdot\|_{r,\mathcal{O}}$ (for $H^r(\mathcal{O})$ and 
$\textbf{H}^r(\mathcal{O})$), $\|\cdot\|_{r,\Gamma}$ (for $H^r(\Gamma)$) and $\|\cdot\|_{L^p(\mathcal{O})}$ (if 
$p\neq 2$). Also, the Hilbert space 
$$\textbf{H}(\dive;\mathcal{O}):=\left\{\textbf{w}\in \textbf{L}^2(\mathcal{O}): \hspace*{0.5cm}
\dive \textbf{w}\in L^2(\mathcal{O})\right\},$$
with norm $\|\cdot\|_{\dive,\mathcal{O}}$, is standard in  the realm of mixed problems (see, e.g. \cite{BF:91}).

On the other hand, the symbol for the $L^2(\Gamma)$ inner product
$$\langle \xi,\lambda\rangle_{\Gamma}:=\int_{\Gamma}\xi\lambda, \forall \xi,\lambda\in L^2(\Gamma),$$
will also be employed for their respective extension as the duality product $H^{-1/2}(\Gamma)\times H^{1/2}(\Gamma)$.
In addition, given two Hilbert spaces $H_1$ and $H_2$, the product space $H_1\times H_2$ will be endowed with 
the norm 
$\|\cdot\|_{H^1\times H^1}=\|\cdot\|_{H_1}+\|\cdot\|_{H_2}$.  Hereafter, given a non-negative integer $k$ and 
a subset $S$ of $\mathbb{R}^2$, $\mathbb{P}^l(S)$ stands for the space of polynomials defined on $S$ of degree 
$\leq l$. Finally, we employed $\textbf{0}$ as a generic null vector. 

The unknowns in the variational formulation of the Navier-Stokes/Darcy coupled problem and the corresponding spaces will be:
$\textbf{u}_S\in \textbf{H}_{\Gamma_S}^1(\Omega_S)$, $p_S\in L^2(\Omega_S)$, 
$\textbf{u}_D\in \textbf{H}_{\Gamma_D}(\dive;\Omega_D)$, $p_D\in L^2(\Omega_D)$, where
\begin{eqnarray*}
 \textbf{H}_{\Gamma_S}^1(\Omega_S)&:=&\left\{\textbf{v}\in \textbf{H}^1(\Omega_S): 
 \hspace*{0.5cm} \textbf{v}=\textbf{0} \mbox{  on  } \Gamma_S\right\},\\
 \textbf{H}_{\Gamma_D}(\dive;\Omega_D)&:=&\left\{\textbf{v}\in\textbf{H}(\dive;\Omega_D): 
 \hspace*{0.5cm} \textbf{v}\cdot\textbf{n}_D=0\mbox{  on  } \Gamma_D\right\}.
 \end{eqnarray*}
In addition, analogously to  \cite{29} we need to define a further unknown on the coupling boundary:
$$\lambda:=p_D\in H^{1/2}(\Sigma).$$
Note that, in principle, the space for $p_D$ does not allow enough regularity for the trace $\lambda$ to exist. 
However, the solution of Darcy equations has the pressure in $H^1(\Omega_D)$.

Next, for the derivation of the weak formulation of (\ref{I1})-(\ref{cd3})  we define the space 
$$L_0^2(\Omega):=\left\{ q\in L^2(\Omega):   \hspace*{0.2cm}\int_{\Omega} q=0\right\},$$
 and we group the unkowns and spaces as follows:\begin{eqnarray*}
 \textbf{u}:=(\textbf{u}_S,\textbf{u}_D)\in \textbf{H}:=
 \textbf{H}_{\Gamma_S}^1(\Omega_S)\times \textbf{H}_{\Gamma_D}(\dive;\Omega_D);
 (p,\lambda)\in\textbf{Q}:=L_0^2(\Omega)\times H^{1/2}(\Sigma),
 \end{eqnarray*}
 where $p:=p_{S}\chi_{\Omega_S}+p_D\chi_{\Omega_D}$, with $\chi_{\Omega_*}$ being the characteristic function for
 $*\in\{S,D\}$.
 
 The weak formulation of the coupled problem (\ref{I1})-(\ref{cd3}) can be stated as follows \cite{MR:2016}:
 Find $(\textbf{u},\psi)=((\textbf{u}_S,\textbf{u}_D),(p,\lambda))\in\textbf{H}\times \textbf{Q}$, such that 
 \begin{eqnarray}\label{FV}
 \left\{
\begin{array}{ccccccccc}\label{r}
  \textbf{a}(\textbf{u}_S;\textbf{u},\textbf{v})+\textbf{b}(\textbf{v},(p,\lambda))&=&\textbf{F}(\textbf{v}), 
  \forall \textbf{v}:=(\textbf{v}_S,\textbf{v}_D)\in\textbf{H},\\\label{A2}
  \textbf{b}(\textbf{u},(q,\xi))&=&0 \forall (q,\xi)\in \textbf{Q},
 \end{array}
\right.
\end{eqnarray}
where $\textbf{a}:\textbf{H}_{\Gamma_S}^1(\Omega_S)\times (\textbf{H}\times \textbf{H})\longrightarrow \mathbb{R}$ and 
$\textbf{b}:\textbf{H}\times \textbf{Q}\longrightarrow\mathbb{R}$ are the forms defined by 
\begin{eqnarray*}
 \textbf{a}(\textbf{w}_S;\textbf{u},\textbf{v})&:=&A_S(\textbf{u}_S,\textbf{v}_S)+O_S(\textbf{w}_S;\textbf{u}_S,\textbf{v}_S)+
 A_D(\textbf{u}_D,\textbf{v}_D),\\
 \textbf{b}(\textbf{v},(q,\xi))&:=&-(q,\dive \textbf{v}_S)_S-(q,\dive \textbf{v}_D)_D+
 \langle \textbf{v}_S\cdot\textbf{n}_S+\textbf{v}_D\cdot\textbf{n}_D,\xi\rangle_{\Sigma},
\end{eqnarray*}
with 
\begin{eqnarray*}
 A_S(\textbf{u}_S,\textbf{v}_S)&:=&2\mu \left(\textbf{e}(\textbf{u}_S),\textbf{e}(\textbf{v}_S)\right)_S+
 \left\langle\frac{\alpha_d\mu}{\sqrt{\tau\cdot\kappa\cdot\tau}}\textbf{u}_S\cdot\tau,\textbf{v}_S\cdot\tau
 \right\rangle_{\Sigma},\\
 O_S(\textbf{w}_S;\textbf{u}_S,\textbf{v}_S)&:=&\rho\left((\textbf{w}_S\cdot\nabla)\textbf{u}_S,\textbf{v}_S\right)_{S},\\
 A_D(\textbf{u}_D,\textbf{v}_D)&:=&(\textbf{K}^{-1}\textbf{u}_D,\textbf{v}_D)_D,
\end{eqnarray*}
and $\textbf{F}(\textbf{v})$ is the linear functional $\textbf{F}: \textbf{H}\longrightarrow \mathbb{R}$ defined as 
$$\textbf{F}(\textbf{v})=(\textbf{f}_S,\textbf{v}_S)_S+(\textbf{f}_D,\textbf{v}_D)_D, \mbox{   } \forall 
\textbf{v}:=(\textbf{v}_S,\textbf{v}_D)\in \textbf{H}.$$

We define the  bilinear form $a_1$ and the nonlinear form $a_2$ by :
\begin{eqnarray}
 [a_1(\textbf{u}),\textbf{v}]&:=&A_S(\textbf{u}_S,\textbf{v}_S)+A_D(\textbf{u}_D,\textbf{v}_D),
 \end{eqnarray}
\begin{eqnarray}
 [a_2(\textbf{u}_S)(\textbf{u}),\textbf{v}]:=O_S(\textbf{u}_S;\textbf{u}_S,\textbf{v}_S),
\end{eqnarray}
and we set, for $\textbf{u}=(\textbf{u}_S,\textbf{u}_D)$, $\textbf{v}=(\textbf{v}_S,\textbf{v}_D)$ and 
$\phi=(q,\xi)$
\begin{eqnarray*}
 [\textbf{A}(\textbf{u}_S)(\textbf{u}_S,\textbf{u}_D),(\textbf{v}_S,\textbf{v}_D)]&:=&
 [a_1(\textbf{u}),\textbf{v}]+[a_2(\textbf{u}_S)(\textbf{u}),\textbf{v}],
 \end{eqnarray*}
\begin{eqnarray*}
 [\textbf{B}(\textbf{v}_S,\textbf{v}_D),\phi]&:=&\textbf{b}(\textbf{v},q)+
 \langle\textbf{v}_S\cdot\textbf{n}_S+\textbf{v}_D\cdot\textbf{n}_D,\xi\rangle_{\Sigma},
 \end{eqnarray*}
 \begin{eqnarray*}
  [\mathcal{F},\textbf{v}]&:=&\textbf{F}(\textbf{v}).
 \end{eqnarray*}
 In all the foregoing terms, $[\cdot,\cdot]$ denotes the duality pairing induced by the corresponding operators.
Then, the formulation (\ref{FV}) is equivalent to, find $(\textbf{u},\psi)\in\textbf{H}\times \textbf{Q}$, with 
$\textbf{u}=(\textbf{u}_S,\textbf{u}_D)$ and $\psi=(p,\lambda)$ such that:
\begin{eqnarray}\label{FVE}
 \left\{
\begin{array}{ccccccccc}\label{r}
 &[\textbf{A}(\textbf{u}_S)(\textbf{u}),\textbf{v}]+[\textbf{B}(\textbf{v}),
 \psi]& &=& &[\mathcal{F},\textbf{v}]& \forall \textbf{v}\in \textbf{H},\\
 &[\textbf{B}(\textbf{u}),\phi]& &=& &0& \forall \phi\in \textbf{Q}.
\end{array}
\right.
\end{eqnarray}

This problem has a unique solution as proved in \cite[Section 2.2]{MR:2016}.
\begin{thm}\cite[Section 2.2, Theorem 2]{MR:2016}\label{ue} Assume that
 $\textbf{f}_S\in \textbf{L}^2(\Omega_S)$ and $\textbf{f}_D\in \textbf{L}^2(\Omega_D)$ satisfy the conditions 
 (36) and (43) of the paper \cite{MR:2016}. Then, there exists a unique solution 
 $(\textbf{u},(p,\lambda))\in \textbf{H}\times \textbf{Q}$ of (\ref{FVE}). In addition, there exists a constant $C> 0$, 
 independent of the solution, such that
 $$\|(\textbf{u},(p,\lambda))\|_{\textbf{H}\times \textbf{Q}}\leq 
 C\left(\|\textbf{f}_S\|_{0,\Omega_S}+\|\textbf{f}_D\|_{0,\Omega_D}\right).$$
\end{thm}
\subsection{Finite element discretization}
Let $\cT_h^S$ and $\cT_h^D$ be respective triangulations of the domains $\Omega_S$ and $\Omega_D$ formed by shape-regular 
triangles of diameter $h_T$ and denote by $h_S$ and $h_D$ their corresponding mesh sizes. Assume that they match on $\Sigma$ 
so that $\cT_h:=\cT_h^S\cup\cT_h^D$ is a triangulation of $\Omega:=\Omega_S\cup\Sigma\cup\Omega_D$. Hereafter 
$h:=\max\{h_S,h_D\}$. 

For each $T\in\cT_h^D$ we consider the local Raviart-Thomas space of the lowest order \cite{RT:77}:
$$RT_0(T):=\mbox{span}\{(1,0),(0,1),(x_1,x_2)\},$$ where $(x_1,x_2)$ is a generic vector in $\mathbb{R}^2$.\\
In addition, for each $T\in\cT_h^S$, we denote by $BR(T)$ the local Bernardi-Raugel space (see \cite{BR:85}):
$$BR(T):=[\mathbb{P}^1(T)]^2\oplus\mbox{span}\{\eta_1\eta_2\textbf{n}_1,\eta_1\eta_3\textbf{n}_2,\eta_1\eta_2\textbf{n}_3\},$$
where $\{\eta_1,\eta_2,\eta_3\}$ are the baricentric coordonates of $T$, and $\{\textbf{n}_1,\textbf{n}_2,\textbf{n}_3\}$ are 
the unit outward normals to opposite sides of the corresponding vertices of $T$. Hence, we define the following finite 
element subspaces:
\begin{eqnarray*}
 \textbf{H}_h(\Omega_S)&:=&\left\{\textbf{v}\in\textbf{H}^1(\Omega_S):\textbf{v}_{|T}\in BR(T), \forall T\in\cT_h^S\right\},\\
 \textbf{H}_h(\Omega_D)&:=&\left\{\textbf{v}\in\textbf{H}(\dive;\Omega_D): \textbf{v}_{|T}\in RT_0(T), \forall T\in \cT_h^D
 \right\},\\
 L_h(\Omega)&:=&\left\{q\in L^2(\Omega): q_{|T}\in\mathbb{P}_0(T), \forall T\in \cT_h\right\}.
\end{eqnarray*}
The finite element subspaces for the velocities and pressure are, respectively, 
\begin{eqnarray*}
 \textbf{H}_{h,\Gamma_S}(\Omega_S)&:=&\textbf{H}_h(\Omega_S)\cap \textbf{H}_{\Gamma_S}^1(\Omega_S),\\
 \textbf{H}_{h,\Gamma_D}(\Omega_D)&:=&\textbf{H}_h(\Omega_D)\cap\textbf{H}_{\Gamma_D}(\dive;\Omega_D),\\
 L_{h,0}(\Omega)&:=&L_h(\Omega)\cap L_0^2(\Omega).
\end{eqnarray*}

In turn, in order to define the discrete spaces for  the unknowns on the interface $\Sigma$, we denote by 
$\Sigma_h$ the partition of $\Sigma$ inherited from $\cT_h^S$ (or $\cT_h^D$) and we assume, without loss of generality, that 
the number of edges of $\Sigma_h$ is even. Then, we let $\Sigma_{2h}$ be the partition of $\Sigma$ arising by joining 
pairs of adjacent edges of $\Sigma_h$. Note that since $\Sigma_h$ is inherited from the interior triangulations, it is 
automatically of bounded variation (i.e., the ratio of lengths of adjacent edges is bounded) and, therefore, so is 
$\Sigma_{2h}$. If the number of edges of $\Sigma_{h}$ is odd, we simply reduce it to the even case by joining any pair of two 
adjacent elements, and then construct $\Sigma_{2h}$ from this reduced partition. Then, we define the following finite 
element subspace for $\lambda\in H^{1/2}(\Sigma)$:
$$\Lambda_h(\Sigma):=\left\{\xi_h\in C^0(\Sigma): {\xi_h}_{|E}\in\mathbb{P}^1(E), \forall E\in \Sigma_{2h}\right\}.$$

In this way, grouping the unknowns and spaces as follows:
\begin{eqnarray*}
 \textbf{u}_h:=(\textbf{u}_{h,S},\textbf{u}_{h,D})\in \textbf{H}_{h,\Gamma_S}(\Omega_S)\times \textbf{H}_{h,\Gamma_D}(\Omega_D),
 (p_h,\lambda_h)\in \textbf{Q}_h:=L_{h,0}(\Omega)\times \Lambda_h(\Sigma),
\end{eqnarray*}
where $p_h:=p_{h,S}\chi_{\Omega_S}+p_{h,D}\chi_{\Omega_D}$, the Galerkin approximation of (\ref{FV}) reads: Find
$(\textbf{u}_h,(p_h,\lambda_h))\in \textbf{H}_h\times Q_h$ such that,
\begin{eqnarray}\label{FVh}
 \left\{
\begin{array}{ccccccccc}\label{r}
 \textbf{a}_h(\textbf{u}_{h,S};\textbf{u}_h,\textbf{v})+\textbf{b}(\textbf{v},(p_h,\lambda_h))&=&\textbf{F}(\textbf{v})
 \mbox{  } \forall \textbf{v}:=(\textbf{v}_S,\textbf{v}_D)\in \textbf{H}_h,\\
 \textbf{b}(\textbf{u}_h,(q,\xi))&=&0\hspace*{0.7cm}\mbox{   } \forall (q,\xi)\in\textbf{Q}_h.
\end{array}
\right.
\end{eqnarray}
Here $\textbf{a}_h: \textbf{H}_{h,\Gamma_S}(\Omega_S)\times(\textbf{H}_h\times \textbf{H}_h)\longrightarrow \mathbb{R}$ is the 
discrete version of $\textbf{a}$ defined by 
\begin{eqnarray}
 \textbf{a}_h(\textbf{w}_S;\textbf{u},\textbf{v})=\textbf{a}(\textbf{w}_S;\textbf{u},\textbf{v})+\textbf{J}_S(\textbf{w}_S;
\textbf{u}_S,\textbf{v}_S),\end{eqnarray}
for all $\textbf{u}_S$, $\textbf{v}_S$, $\textbf{w}_S\in\textbf{H}_h(\Omega_S),$ and where 
$\textbf{J}_S(\textbf{w}_S;
\textbf{u}_S,\textbf{v}_S)$ is defined by:
$$\textbf{J}_S(\textbf{w}_S;
\textbf{u}_S,\textbf{v}_S) :=\frac{\rho}{2}\left( \textbf{u}_S\dive\textbf{w}_S,\textbf{v}_S\right)_S.$$
 As before, we set, for $\textbf{u}_h=(\textbf{u}_{h,S},\textbf{u}_{h,D})$, 
 $\textbf{v}_h=(\textbf{v}_{h,S},\textbf{v}_{h,D})$ and 
$\phi_h=(q_h,\xi_h)$
\begin{eqnarray*}
 [\textbf{A}_h(\textbf{u}_{h,S})(\textbf{u}_{h,S},\textbf{u}_{h,D}),(\textbf{v}_{h,S},\textbf{v}_{h,D})]&:=&
 [a_1(\textbf{u}_h),\textbf{v}_h]+[a_2^h(\textbf{u}_{h,S})(\textbf{u}_h),\textbf{v}_h],
 \end{eqnarray*}
 with, 
 $$[a_2^h(\textbf{u}_{h,S})(\textbf{u}_h),\textbf{v}_h]:=
 [a_2(\textbf{u}_{h,S})(\textbf{u}_h),\textbf{v}_h]+\textbf{J}_S(\textbf{u}_{h,S};
\textbf{u}_{h,S},\textbf{v}_{h,S}).$$
Thus, the formulation (\ref{FVh}) is equivalent to, find $(\textbf{u}_h,\psi_h)\in\textbf{H}_h\times \textbf{Q}_h$, 
with \\$\textbf{u}_h=(\textbf{u}_{h,S},\textbf{u}_{h,D})$ and $\psi_h=(p_h,\lambda_h)$ such that:
\begin{eqnarray}\label{FVEh}
 \left\{
\begin{array}{ccccccccc}\label{r}
 &[\textbf{A}_h(\textbf{u}_{h,S})(\textbf{u}_{h}),
 \textbf{v}]+[\textbf{B}(\textbf{v}_{h}),
 \psi_h] & = [\mathcal{F},\textbf{v}_h] \forall \textbf{v}_h\in \textbf{H}_h,\\
 &[\textbf{B}(\textbf{u}_{h}),\phi_h]&= 0 \forall \phi_h\in \textbf{Q}_h.
\end{array}
\right.
\end{eqnarray}
\begin{thm}(See \cite[Section 3.2, Theorem 4 and Theorem 6]{MR:2016})\label{uh}
Assume that $\textbf{f}_S\in\textbf{L}^2(\Omega_S)$ and $\textbf{f}_D\in\textbf{L}^2(\Omega_D)$ satisfy the 
conditions $(71)$, $(78)$, $(82)$ and  $(86)$ of the reference \cite{MR:2016}. Then,
 there exists a unique solution $(\textbf{u}_h,(p_h,\lambda_h))\in\textbf{H}_h\times \textbf{Q}_h$ to problem (\ref{FVEh}) and if the solution 
 $(\textbf{u},(p,\lambda))\in\textbf{H}\times \textbf{Q}$ of the continuous problem (\ref{FVE}) is smooth enough, then 
 we have: 
 \begin{eqnarray*}
  \|(\textbf{u},(p,\lambda))-(\textbf{u}_h,(p_h,\lambda_h))\|_{\textbf{H}\times \textbf{Q}}&\lesssim& 
  h\left(\|\textbf{u}_S\|_{2,\Omega_S}+\|\textbf{u}_D\|_{1,\Omega_D}+
  \|\dive \textbf{u}_D\|_{1,\Omega_D}\right.\\
  &+&\|p\|_{1,\Omega}
  +\left.\|\lambda\|_{3/2,\Sigma}\right).
 \end{eqnarray*}

\end{thm}
Here and below, in order to avoid excessive use of constants, the abbreviation $x\lesssim y$ stand for 
$x\leq cy$, with $c$ a positive constant independent of $x$, $y$ and $\cT_h$.

For $\textbf{v}=(\textbf{v}_S,\textbf{v}_D)\in\textbf{H}_h$ and for $(q,\xi)\in \textbf{Q}_h$, we can subtract (\ref{FV}) to (\ref{FVh}) to obtain the Galerkin
orthogonality relation:
\begin{eqnarray*}
 A_s(\textbf{e}_{\textbf{u}_S},\textbf{v}_S)+A_D(\textbf{e}_{\textbf{u}_D},\textbf{v}_D)+
 O_S^h(\textbf{u}_S;\textbf{u}_S,\textbf{v}_S)
 -O_S^h(\textbf{u}_{h,S};\textbf{u}_{h,S},\textbf{v}_S)+\textbf{b}(e_p,e_{\lambda})&=&0,\\
 \textbf{b}((\textbf{e}_{\textbf{u}_S,\textbf{e}_{\textbf{u}_D}}),(q,\xi))&=&0,
\end{eqnarray*}
where here and below, the errors in the velocity, in the pressure and in the Lagrange multiplier are respectively 
defined by 
\begin{eqnarray}
 \textbf{e}_{\textbf{u}_*}:=\textbf{u}_{*}-\textbf{u}_{h,*}; \hspace*{0.4cm} 
 e_p=p-p_h \mbox{  and  } e_{\lambda}=\lambda-\lambda_h, *\in\{S,D\}.
\end{eqnarray}

We end this section with  some notation again. For each $T\in\cT_h$, we denoted by $\cE(T)$ (resp. $\mathcal{N}(T)$) the set of its edges (resp. vertices) and set 
$\cE_h=\displaystyle\bigcup_{T\in\cT_h}\cE(T)$, $\cN_h=\displaystyle\bigcup_{T\in\cT_h}\mathcal{N}(T)$. For 
$\mathcal{A}\subset \bar{\Omega}$ we define 
\begin{eqnarray*}
 \cE_h(\mathcal{A}):=\left\{E\in\cE_h: E\subset \mathcal{A}\right\} \mbox{  and  } 
 \mathcal{N}_h(\mathcal{A}):=\left\{\textbf{x}\in\mathcal{N}_h, \textbf{x}\in\mathcal{A}\right\}.
\end{eqnarray*}
With every edge $E\in\cE_h$, we introduce the outer normal vector by $\textbf{n}=(n_x,n_y)^\top$. Furthermore, 
for each face $E$, we fix one of the two normal vectors and denote it by $\textbf{n}_E$. In addition, we introduce
the tangente vector $\tau=\textbf{n}^\top:=(-n_y,n_x)^\top$ such that it is oriented positively (with respect to $T$). 
Similarly, set $\tau_E:=\textbf{n}_E^{\top}$. 

For any $E\in\cE_h$
and any piecewise continuous function $\varphi$, 
we denote by $[\varphi]_E$ its jump   across $E$ in the direction of $\textbf{n}_E$:
\begin{eqnarray*}
 [\varphi]_E(x):=
 \left\{
\begin{array}{cccccc}\label{r}
&\displaystyle\lim_{t\rightarrow 0+} \varphi(x+t\textbf{n}_E)-\lim_{t\rightarrow 0+} \varphi (x-t\textbf{n}_E) &
&\mbox{for an interior edge/face $E$,}&\\
&- \displaystyle\lim_{t\rightarrow 0+}\varphi (x-t\textbf{n}_E)& &\mbox{for a boundary edge/face $E$}.&
\end{array}
\right.
\end{eqnarray*}

Furthermore one requires local subdomains (also known as patches). As usual, let $w_T$ be the union of all elements having a common face 
with $T$. Similarly let $w_E$ be the union of both elements having $E$ as face (with appropriate modifications for a boundary face).
By $w_{\textbf{x}}$ we denote the union of all elements having $\textbf{x}$ as node.

In the sequel we will also make use of the
following differential operator:
$$\curl \textbf{v}:=\frac{\partial v_2}{\partial x_1}-\frac{\partial v_1}{\partial x_2} \mbox{  for  }
\textbf{v}=(v_1,v_2).$$

\section{Error estimators}\label{sec:r2}
In order to solve the Navier-Stokes/Darcy coupled problem by efficient adative finite element methods, reliable and efficient
a posteriori error analysis is important to provide appropriated indicators. In this section, we first define the local and 
global indicators and then the upper and lower error bounds are derived. 

\subsection{Residual error estimator}
The general philosophy of residual error estimators is to estimate an appropriate norm of the correct residual by terms that 
can be  evaluated easier, and that involve the data at hand. To this end denote the exact element residuals by 
\begin{eqnarray*}
 \textbf{R}_{S,T}&=&\textbf{f}_S+2\mu\dive(\textbf{e}(\textbf{u}_{h,S}))-
 \nabla p_{h,S}-\rho\left(\textbf{u}_{h,S}\cdot\nabla\right)\textbf{u}_{h,S}-\frac{\rho}{2}
 \textbf{u}_{h,S}\dive\textbf{u}_{h,S} \mbox{  in  } T\in \mathcal{T}_h^S,\\
 \textbf{R}_{D,T}&=& \textbf{f}_D-\textbf{K}^{-1}\textbf{u}_{h,D}-\nabla p_{h,D} \mbox{  in } T\in \mathcal{T}_h^D.
\end{eqnarray*}
As it is common, these exact residuals are replaced by some finite-dimensional approximation called approximate element 
residual $\textbf{r}_{*,T}$, $*\in\{S,D\}$, 
$$ \textbf{r}_{*,T}\in[\mathbb{P}^k(T)]^2 \mbox{  on  } T\in \mathcal{T}_h^*.$$
This approximation is here  achieved by  projecting 
 $\textbf{f}_S$ on the space of piecewise constant functions in $\Om_S$
and piecewise ${\mathbb P}^1$ functions  in $\Om_D$, more precisely 
for all $T\in \mathcal{T}_h^S$, we take
\[
\bff_{T,S}=\frac{1}{|T|}\int_T \bff(x)\,dx,
\]
while for all $T\in \mathcal{T}_h^D$, we take $\bff_{T,D}$ as the unique element of $[{\mathbb P}^1(T)]^2$
such that
\[
\int_T\bff_{T,D}(x)\cdot\textbf{q}(x)\,dx = \int_T \bff(x)\cdot\textbf{q}(x)\,dx, \forall \textbf{q}\in [{\mathbb P}^1(T)]^2.
\]
Finally the global function   $\textbf{f}_h$ is defined by
\[
\bff_{h,*}= \bff_{T,*} \hbox{ in } T,  \forall T\in \mathcal{T}_h^*, *\in\{S,D\}.
\]
Hence 
\begin{eqnarray*}
 \textbf{r}_{S,T}&=&\textbf{f}_{T,S}+2\mu\dive \textbf{e}(\textbf{u}_{h,S})-\nabla p_{h,S} 
 -\rho(\textbf{u}_{h,S}\cdot\nabla)\textbf{u}_{h,S}-\frac{\rho}{2}
 [\dive(\textbf{u}_{h,S})\textbf{u}_{h,S}]\mbox{  in  } 
 T\in \mathcal{T}_h^S,\\
 \textbf{r}_{D,T}&=& \textbf{f}_{T,D}-\textbf{K}^{-1}\textbf{u}_{h,D}-\nabla p_{h,D} \mbox{  in } T\in \mathcal{T}_h^D.
\end{eqnarray*}
Next introduce the gradient jump in normal direction by 
\begin{eqnarray*}
\textbf{J}_{E,\textbf{n}_E}&:=&
 \left\{
 \begin{array}{ccccccccc}
  &[(2\mu\textbf{e}(\textbf{u}_{h,S})-p_{h,S}\textbf{I})\cdot\textbf{n}_E]_E&  &\mbox{  if  } &
  &E\in\cE_h(\Omega_S),&\\
  &\textbf{0}& &\mbox{  if  }&  &E\in\cE_h(\partial\Omega_S),&
 \end{array}
\right.
\end{eqnarray*}
where $\textbf{I}$ is the identity matrix of $\mathbb{R}^{2\times 2}$.
\begin{dfn} (\textbf{Residual error estimator}) The residual error estimator is globally  defined by: 
\begin{eqnarray}
 \Theta:=\left\{\displaystyle\sum_{T\in\cT_h^S}\Theta_{S,T}^2+\displaystyle\sum_{T\in\cT_h^D}\Theta_{D,T}^2\right\}^{1/2},
\end{eqnarray}
where the local error indicators $\Theta_{S,T}^2$ (with $T\in\cT_h^S$) and $\Theta_{D,T}^2$ (with $T\in\cT_h^D$) are given 
by 
\begin{eqnarray}\nonumber
 \Theta_{S,T}^2&:=&\|\textbf{r}_{S,T}\|_{0,T}^2+
 \displaystyle\sum_{E\in\cE(T)\cap\cE_h(\bar{\Omega}_S)}\|\textbf{J}_{E,\textbf{n}_E}\|_{0,E}^2\\\nonumber
 &+&
 \displaystyle\sum_{E\in\cE(T)\cap\cE_h(\bar{\Sigma})}
   \|-p_{h,S}+p_{h,D}-2\mu\textbf{n}_S\cdot\textbf{e}(\textbf{u}_{h,S})\cdot\textbf{n}_S\|_{0,E}^2\\\label{theta1}
   &+&
   \displaystyle\sum_{E\in\cE(T)\cap\cE_h(\bar{\Sigma})}
   \left\|\frac{\alpha_d\mu}{\sqrt{\tau\cdot\kappa\cdot\tau}}\textbf{u}_{h,S}\cdot\tau+2\mu\textbf{n}_S\cdot
 \textbf{e}(\textbf{u}_{h,S})\cdot\tau\right\|_{0,E}^2\\\nonumber
 &+&\displaystyle\sum_{E\in\cE_h(\bar{\Sigma})}
 \|\textbf{u}_{h,S}\cdot\textbf{n}_S+\textbf{u}_{h,D}\cdot\textbf{n}_D\|_{0,E}^2\\\nonumber
  &+&\|\dive\textbf{u}_{h,S}\|_{0,T}^2,
\end{eqnarray}
and 
\begin{eqnarray}\nonumber\label{theta2}
 \Theta_{D,T}^2&:=& h_T^2\left(\|\textbf{f}_{h,D}-\nabla p_{h,D}-\textbf{K}^{-1}\textbf{u}_{h,D}\|_{0,T}^2+
  \|\curl (\textbf{f}_{h,D}-\textbf{K}^{-1}\textbf{u}_{h,D})\|_{0,T}^2\right)\\\nonumber
  &+&
  \displaystyle\sum_{E\in\cE(T)\cap\cE_h(\Omega_D)}h_E
  \|[ (\textbf{f}_{h,D}-\textbf{K}^{-1}\textbf{u}_{h,D}-\nabla p_{h,D})\cdot\tau_E]_E\|_{0,E}^2\\
  &+&
  \displaystyle\sum_{E\in\cE(T)\cap\cE_h(\partial\Omega_D)}h_E
  \|(\textbf{f}_{h,D}-\textbf{K}^{-1}\textbf{u}_{h,D}-\nabla p_{h,D})\cdot\tau_E\|_{0,E}^2\\\nonumber
  &+&\displaystyle\sum_{E\in\cE(T)\cap\cE_h(\bar{\Sigma})}h_E
 \|p_{h,D}-\lambda_h\|_{0,E}^2\\\nonumber
 &+&\|\dive\textbf{u}_{h,D}\|_{0,T}^2.
\end{eqnarray}
\end{dfn}
Furthermore denote the local and global approximation terms by
\begin{eqnarray*}
 \zeta_{T}:=
\left\{
 \begin{array}{cc}
  \parallel \textbf{f}_S-\textbf{f}_{h,S}\parallel_{0,T}
&\forall T\in \mathcal{T}_h^S,
\\
 h_{T}(\parallel \textbf{f}_D-\textbf{f}_{h,D}\parallel_{0,T}+
\|\curl (\textbf{f}_D-\textbf{f}_{h,D})\|_{0,T})
&\forall T\in \mathcal{T}_h^D,
\end{array}
\right.
\end{eqnarray*}
\begin{eqnarray}
 \zeta &:=&\left(\zeta_S^2+\zeta_D^2\right)^{1/2},
\end{eqnarray}
where
\begin{eqnarray}
 \zeta_S:=\left(\displaystyle\sum_{T\in\cT_h^S}\zeta_T^2\right)^{1/2} \mbox{  and  } 
 \zeta_D:=\left(\displaystyle\sum_{T\in\cT_h^D}\zeta_T^2\right)^{1/2}.
\end{eqnarray}

\begin{rmq}
 The residual character of each term on the right-hand sides of (\ref{theta1}) and (\ref{theta2}) is quite
clear since if $(\textbf{u}_h,(p_h,\lambda_h))$ would be the exact solution of (\ref{FVE}), then they would vanish.
\end{rmq}
\subsection{Reliability of the a posteriori error estimator} \label{sec:up}
Recall the notation for the velocity error $\textbf{e}_{\textbf{u}} = \textbf{u}-\textbf{u}_h$, the
pressure error $e_p=p-p_h$ and the Lagrange multiplier $e_{\lambda}=\lambda-\lambda_h$.
The a posteriori error estimator $\Theta$ is consider 
reliable if it satisfies 
\begin{eqnarray}\label{Upper}
 \|(\textbf{e}_{\textbf{u}},(e_p,e_{\lambda}))\|_{\textbf{H}\times \textbf{Q}}&\lesssim&  \Theta+\zeta.
\end{eqnarray}
In this subsection, we shall prove this estimate. But before, we remains some analytical tools. We introduce the 
Cl\'ement interpolation operator 
$I_h^*:H^1(\Omega_*)\longrightarrow H_h^1(\Omega_*),$ with 
$$H_h^1(\Omega_*):=\left\{v\in C^0(\bar{\Omega}_*):  
v_{|T}\in \mathbb{P}^1(T), \forall T\in \cT_h^*\right\},$$
approximates optimally non-smooth functions by continuous piecewise linear functions. 
In addition, we will make use of a vector valued version of $I_h^*$, that is, 
$\textbf{I}_h^*: \textbf{H}^1(\Omega_*)\longrightarrow \textbf{H}_h^1(\Omega_*)$, which is
defined componentwise by $I_h^*$. The following lemma establishes the  local approximation properties of 
$I_h^*$ (and hence of $\textbf{I}_h^*$), for proof see \cite[Section 3]{clement:75}.

\begin{lem} (\textbf{Cl\'ement operator})\label{clement}
 For each $*\in \{S,D\}$ there exist constants $c_1$,  $c_2> 0$, independent of $h$, such that for all 
 $v\in H^1(\Omega_*)$ there holds
 \begin{eqnarray}
  \|v-I_h^*v\|_{0,T}&\leq&c_1h_T\|v\|_{1,\Delta_*(T)}, \forall T\in \cT_h^*,
 \end{eqnarray}
 and 
 \begin{eqnarray}
  \|v-I_h^*v\|_{0,E}&\leq&c_2h_E^{1/2}\|v\|_{1,\Delta_*(E)}, \forall T\in \cT_h^*,\forall E\in\cE_h,
 \end{eqnarray}
 where 
 \begin{eqnarray*}
  \Delta_*(T):=\cup\left\{T'\in\cT_h^*: T'\cap T\neq\emptyset\right\}\mbox{      and     } 
  \Delta_*(E):=\cup\left\{T'\in\cT_h^*: T'\cap T\neq\emptyset \right\}.
 \end{eqnarray*}
\end{lem}

Proceeding analogously to \cite[Section 2.5]{CGOS:2015} (see also \cite{SGR:2016}), we  first  let 
$\textbf{P}:\textbf{X}:=\textbf{H}\times \textbf{Q}\longrightarrow \textbf{X}':=\textbf{H}'\times \textbf{Q}'$ and 
$\textbf{P}_h:\textbf{X}_h:=\textbf{H}_h\times \textbf{Q}_h\longrightarrow \textbf{X}_h':=\textbf{H}_h'\times \textbf{M}_h'$ 
be 
the nonlinear operator  suggested by the left hand sides of (\ref{FV}) and (\ref{FVh}) with the velocity solutions 
$\textbf{u}_{S}\in \textbf{H}_{\Gamma_S}^1(\Omega_S)$ and $\textbf{u}_{S,h}\in\textbf{H}_{h,\Gamma_S}(\Omega_S)$, that is 
\begin{eqnarray}\label{P}
 [\textbf{P}(\textbf{U}),\textbf{V}]:=[(a_1+a_2(\textbf{u}_S))(\textbf{u}),\textbf{v}]+
 [\textbf{B}(\textbf{u}_S,\textbf{u}_D),\phi]+[\textbf{B}(\textbf{v}_S,\textbf{v}_D),\psi],
\end{eqnarray}
for all $\textbf{U}=((\textbf{u}_S,\textbf{u}_D),\psi)$, $\textbf{V}=((\textbf{v}_S,\textbf{v}_D),\phi)$, where 
$\psi=(p,\lambda)$, $\phi=(q,\xi)$; and 
\begin{eqnarray}\label{Ph}
 [\textbf{P}_h(\textbf{U}_h),\textbf{V}_h]&:=&[a_1+a_2^h(\textbf{u}_{h,S})(\textbf{u}_{h,S},\textbf{u}_{h,D})]+
 \left[\textbf{B}(\textbf{u}_{h,S},\textbf{u}_{h,D}),\phi_h\right]\\\nonumber
 &+&[\textbf{B}(\textbf{v}_{h,S},\textbf{v}_{h,D}),\psi_h],
\end{eqnarray}
for all $\textbf{U}_h=((\textbf{u}_{h,S},\textbf{u}_{h,D}),\psi_h)$, $\textbf{V}_h=
((\textbf{v}_{h,S},\textbf{v}_{h,D}),\phi_h)$.
Then, setting, \\$\mathcal{F}:=(\textbf{F},\textbf{O})\in \textbf{H}'\times \textbf{Q}'$ with $\textbf{O}\equiv 0 
\mbox{  on  } \textbf{H}\times \textbf{Q}$,  it is clear from 
(\ref{FVE}) and (\ref{FVEh}) that $\textbf{P}$ and $\textbf{P}_h$ satisfy
\begin{eqnarray}
 [\textbf{P}(\textbf{U}),\textbf{V}]=[\mathcal{F},\textbf{V}], \forall \textbf{V}\in \textbf{H}\times \textbf{Q}
\end{eqnarray}
and 
\begin{eqnarray}
 [\textbf{P}_h(\textbf{U}_h),\textbf{V}_h]=[\mathcal{F},\textbf{V}_h], \forall \textbf{V}_h\in \textbf{H}_h\times \textbf{Q}_h,
\end{eqnarray}
respectively. In addition, we find, as explained in 
\cite[Section 5.2]{CGOS:2015}, that  $a_1$ has hemi-continuous first order G\^ateaux derivative 
$\mathcal{D}_{a_1}: \textbf{X}\longrightarrow \mathcal{L}(\textbf{X},\textbf{X}')$. Is this way, the G\^ateaux 
derivative of $\textbf{P}$ at $\textbf{W}\in \textbf{X}$ is obtained by replacing 
$[a_1(\cdot),\cdot]$ in (\ref{P}) by $\mathcal{D}_{a_1}(\textbf{W})(.,.)$ (see \cite[Section 5.2]{CGOS:2015} for details), 
that is 
\begin{eqnarray*}
 \mathcal{D}\textbf{P}(\textbf{W})(\textbf{U},\textbf{V}):=
 \mathcal{D}_{a_1}(\textbf{w})(\textbf{u},\textbf{v})+[a_2(\textbf{u}_S)(\textbf{u}),\textbf{v}]+
 [\textbf{B}(\textbf{v}_S,\textbf{v}_D),\psi]+[\textbf{B}(\textbf{u}_S,\textbf{u}_D),\phi], 
\end{eqnarray*}
 for all $\textbf{U}=((\textbf{u}_S,\textbf{u}_D),\psi)$, $\textbf{V}=((\textbf{v}_S,\textbf{v}_D),\phi)\in 
 \textbf{H}\times \textbf{Q}$. We deduce (see \cite[Section 5.2]{CGOS:2015}) the existence of a positive constant $C_{\textbf{P}}$, independent of 
 $\textbf{W}$ and the continuous and discrete solutions, such that the following global inf-sup condition holds, 
 \begin{eqnarray}
  C_{\textbf{P}}\|\textbf{U}\|_{\textbf{H}\times \textbf{Q}}\leq \sup_{\textbf{V}\in (\textbf{H}\times \textbf{Q})^*}
  \frac{\mathcal{D}\textbf{P}(\textbf{W})(\textbf{U},\textbf{V})}{\|\textbf{V}\|_{\textbf{H}\times \textbf{Q}}}, 
  \forall  \textbf{W}\in \textbf{H}\times \textbf{Q}.
 \end{eqnarray}
 
 We are now in position of establishing the following preliminary a posteriori error estimate.

\begin{thm}
 The following estimation holds,
 \begin{eqnarray}\label{PUP}
  \|\textbf{U}-\textbf{U}_h\|_{\textbf{H}\times \textbf{Q}}\lesssim \|R\|_{(\textbf{H}\times \textbf{Q})'},
 \end{eqnarray}
where $R:\textbf{H}\times \textbf{Q}\longrightarrow \mathbb{R}$ is the residual functional given by 
$R(\textbf{V}):=[\mathcal{F}-\textbf{P}_h(\textbf{U}_h),\textbf{V}]$, for all $\textbf{V}\in\textbf{H}\times \textbf{Q}$, which 
satisfies 
\begin{eqnarray}\label{Residu}
R(\textbf{V}_h)=0, \forall \textbf{V}_h\in \textbf{H}_h\times \textbf{Q}_h.
\end{eqnarray}
\end{thm}
\begin{proof}
 The proof is similar to \cite[Page 955, proof of Theorem 3.5]{SGR:2016}. Further details are omitted.
\end{proof}
According to the upper bound (\ref{PUP}) provided by the previous theorem, it only remains now to 
estimate 
$\|R\|_{(\textbf{H}\times \textbf{Q})'}$. To this end, we first observe that the functional $R$ can be 
decomposed as follows:
\begin{eqnarray}
 R(\textbf{V}):=R_1(\textbf{v}_{S})+R_2(\textbf{v}_D)+R_3(q_S)+R_4(q_D)+R_5(\xi)+R_6(\textbf{v}),
\end{eqnarray}
for all $\textbf{V}=(\textbf{v},(q,\xi))\in\textbf{H}\times \textbf{Q}$, with 
$\textbf{v}=(\textbf{v}_S,\textbf{v}_D)$, $q=(q_S,q_D)$; and  where, 
\begin{eqnarray*}
 R_1(\textbf{v}_{S})&:=&\left(\textbf{f}_S+2\mu\dive(\textbf{e}(\textbf{u}_{h,S}))-
 \nabla p_{h,S}-\rho\left(\textbf{u}_{h,S}\cdot\nabla\right)\textbf{u}_{h,S}-\frac{\rho}{2}
 \textbf{u}_{h,S}\dive\textbf{u}_{h,S},\textbf{v}_S\right)_S\\
 &+&\left([p_{h,S}\textbf{I}-2\mu\textbf{e}(\textbf{u}_{h,S}))\cdot\textbf{n}_S],\textbf{v}_S\right)_{\partial\Omega_S}\\
 &+&\left\langle-p_{h,S}+p_{h,D}-2\mu\textbf{n}_S\cdot\textbf{e}(\textbf{u}_{h,S})\cdot\textbf{n}_S,
 \textbf{v}_S\cdot\textbf{n}_S\right\rangle_{\Sigma}\\
 &-&\left\langle\frac{\alpha_d\mu}{\sqrt{\tau\cdot\kappa\cdot\tau}}\textbf{u}_{h,S}\cdot\tau+2\mu\textbf{n}_S\cdot
 \textbf{e}(\textbf{u}_{h,S})\cdot\tau,\textbf{v}_S\cdot\tau\right\rangle_{\Sigma},\\
 R_2(\textbf{v}_D)&:=&\left(\textbf{f}_D-\nabla p_{h,D}-\textbf{K}^{-1}\textbf{u}_{h,D},\textbf{v}_D\right)_D,\\
 R_3(q_S)&:=&\left(q_S,\nabla\cdot\textbf{u}_{h,S}\right)_S,\\
 R_4(q_D)&:=&\left(q_D,\nabla\cdot\textbf{u}_{h,D}\right)_D,\\
 R_5(\xi)&:=&-\langle\textbf{u}_{h,S}\cdot\textbf{n}_S+\textbf{u}_{h,D}\cdot\textbf{n}_D,\xi\rangle_{\Sigma},\\
 R_6(\textbf{v}_D)&:=&\left\langle p_{h,D}-\lambda_h,\textbf{v}_S\cdot\textbf{n}_S+\textbf{v}_D\cdot\textbf{n}_D
 \right\rangle_{\Sigma}.
\end{eqnarray*}
In this way, it follows that 
\begin{eqnarray}\label{R}\nonumber
 \|R\|_{(\textbf{H}\times\textbf{Q})'}&\leq& \left\{\|R_1\|_{\textbf{H}_{\Gamma_S}(\Omega_S)'}+
 \|R_2\|_{\textbf{H}(\dive; \Omega_D)'}+\|R_3\|_{L^2(\Omega_S)'}+\|R_4\|_{L^2(\Omega_D)'}\right.\\\label{R}
 &+&\|R_5\|_{H^{-1/2}(\Sigma)}
 +\left.\|R_6\|_{H^{-1/2}(\Sigma)}\right\}.
\end{eqnarray}
Hence, our next purpose is to derive suitable upper bounds for each one of the terms on the right hand side of (\ref{R}). 
We start with the following lemma, which is a direct consequence of the Cauchy-Schwarz inequality and the trace inequality. 
\begin{lem}\label{Lemma1}
 The following estimation holds:
 \begin{eqnarray*}
   \|R_1\|_{\textbf{H}_{\Gamma_S}(\Omega_S)'}
  & \lesssim&\left\{\displaystyle\sum_{T\in\cT_h^S}\Big(\|\textbf{r}_{S,T}\|_{0,T}^2+
    \displaystyle\sum_{E\in\cE(T)\cap\cE_h(\bar{\Omega}_S)}\|\textbf{J}_{E,\textbf{n}_E}\|_{0,E}^2\right.\\
   &+&\left.  \displaystyle\sum_{E\in\cE(T)\cap\cE_h(\bar{\Sigma})}
   \|-p_{h,S}+p_{h,D}-2\mu\textbf{n}_S\cdot\textbf{e}(\textbf{u}_{h,S})\cdot\textbf{n}_S\|_{0,E}^2\right.\\
   &+&\left. \displaystyle\sum_{E\in\cE(T)\cap\cE_h(\bar{\Sigma})}
   \left\|\frac{\alpha_d\mu}{\sqrt{\tau\cdot\kappa\cdot\tau}}\textbf{u}_{h,S}\cdot\tau+2\mu\textbf{n}_S\cdot
 \textbf{e}(\textbf{u}_{h,S})\cdot\tau\right\|_{0,E}^2
 \Big)
 \right\}^{1/2}\\
 &+&
 \zeta_S.
 \end{eqnarray*}
In addition, there holds:
\begin{eqnarray*}
 \|R_3\|_{L^2(\Omega_S)'}&\lesssim& \left\{\displaystyle\sum_{T\in\cT_h^S}\|\dive\textbf{u}_{h,S}\|_{0,T}^2\right\}^{1/2},\\
 \|R_4\|_{L^2(\Omega_D)'}&\lesssim& \left\{\displaystyle\sum_{T\in\cT_h^D}\|\dive\textbf{u}_{h,D}\|_{0,T}^2\right\}^{1/2},\\
 \|R_5\|_{H^{-1/2}(\Sigma)}&\lesssim& \left\{\displaystyle\sum_{E\in\cE_h(\bar{\Sigma})}
 \|\textbf{u}_{h,S}\cdot\textbf{n}_S+\textbf{u}_{h,D}\cdot\textbf{n}_D\|_{0,E}^2\right\}^{1/2}.
\end{eqnarray*}
\end{lem}

Next, we derive the upper error bound for $R_2$ and $R_6$. We have the following lemma:

\begin{lem}\label{Lemma2}
 There holds:
 \begin{eqnarray}\nonumber
  \|R_2\|_{\textbf{H}(\dive,\Omega_D)'}\lesssim \left\{\displaystyle\sum_{T\in\cT_h^D} h_T^2\Big(
  \|\textbf{f}_{h,D}-\nabla p_{h,D}-\textbf{K}^{-1}\textbf{u}_{h,D}\|_{0,T}^2+
  \|\curl (\textbf{f}_{h,D}-\textbf{K}^{-1}\textbf{u}_{h,D})\|_{0,T}^2\Big)\right.\\\label{R2}
  \left. +\displaystyle\sum_{E\in\cE(T)\cap\cE_h(\Omega_D)}h_E
  \|[ (\textbf{f}_{h,D}-\textbf{K}^{-1}\textbf{u}_{h,D}-\nabla p_{h,D})\cdot\tau_E]_E\|_{0,E}^2\right.\\\nonumber
  \left. +\displaystyle\sum_{E\in\cE(T)\cap\cE_h(\partial\Omega_D)}h_E
  \|(\textbf{f}_{h,D}-\textbf{K}^{-1}\textbf{u}_{h,D}-\nabla p_{h,D})\cdot\tau_E\|_{0,E}^2
  \right\}^{1/2}
  +
  \zeta_D,
 \end{eqnarray}
 and 
 \begin{eqnarray}\label{R6}
  \|R_6\|_{H^{-1/2}(\Sigma)}\lesssim \left\{\displaystyle\sum_{E\in\cE_h(\bar{\Sigma})}h_E\|p_{h,D}-
  \lambda_h\|_{0,E}^2\right\}^{1/2}.
 \end{eqnarray}

\end{lem}
\begin{proof}
 The estimate (\ref{R6}) follows directly from \cite[Page 953, Lemma 3.4]{SGR:2016}. Our next goal is to derive 
 the upper bound for $R_2$, for which, given $\textbf{v}_D\in\textbf{H}(\dive;\Omega_D)$, we consider its Helmholtz 
 decomposition provided in \cite[Page 1882, Lemma 3.3]{47}. More precisely, there is $C_D> 0$ such that each 
 $\textbf{v}_D\in\textbf{H}(\dive;\Omega_D)$ can be
decomposed as $\textbf{v}_D=\textbf{w}_D+\curl \beta_D$, where $\textbf{w}_D\in \textbf{H}^1(\Omega_D)$ and $\beta_D\in 
H^1(\Omega_D)$ with $\int_{\Omega_D}\beta_D=0$, and 
$\|\textbf{w}\|_{\textbf{H}^1(\Omega_D)}+\|\beta\|_{H^1(\Omega_D)}\leq C_D\|\textbf{v}_D\|_{\textbf{H}(\dive;\Omega_D)}$.\\
Then, defining $\textbf{v}_{h,D}:=I_h^D(\textbf{w}_D)+\curl (I_h^D\beta_D)\in \textbf{H}_h(\dive,\Omega_D)$, which 
can be seen as a discrete Helmholtz decomposition of $\textbf{v}_{h,D}$, and applying from (\ref{Residu}) that
$R_2(\textbf{v}_{h,D})=0$, we can write 
\begin{eqnarray}
 R_2(\textbf{v}_{D})=R_2(\textbf{v}_{D}-\textbf{v}_{h,D})=
 R_2(\textbf{w}_D-\textbf{w}_{h,D})+R_2(\curl(\beta_D-\beta_{h,D})),
\end{eqnarray}
with $\textbf{w}_{h,D}=I_h^D(\textbf{w}_D)$ and $\beta_{h,D}=\curl (I_h^D\beta_D)$. 
Note that $\curl (\nabla p_{h,D})=0$. Thus,  we have now by standard Green's formula in two spatial dimensions 
on each $T$, the inequality:
\begin{eqnarray*}
 R_2(\textbf{v}_{D})\lesssim\displaystyle\sum_{T\in\cT_h^D}\left\{
 (\textbf{f}_{D}-\nabla p_{h,D}-\textbf{K}^{-1}\textbf{u}_{h,D},
 \textbf{w}_D-\textbf{w}_{h,D})_T\right.\\
 -\left.\left(\curl(\textbf{f}_D-\textbf{K}^{-1}\textbf{u}_{h,D}), \beta_D-\beta_{h,D}\right)_T\right.\\
 +\left.
 ((\textbf{f}_D-\nabla p_{h,D}-\textbf{K}^{-1}\textbf{u}_{h,D})\cdot\tau_D,\beta_D-\beta_{h,D})_{\partial T}
 \right\}.
 \end{eqnarray*}
 
We introduce the approximation $\textbf{f}_{h,D}$ of $\textbf{f}_D$, and then
the estimate (\ref{R2}) follows by applying Cauchy-Schwarz inequality and
the Cl\'ement operator estimations of Lemma \ref{clement}.
\end{proof}
 We have now the main result of this subsection:
 \begin{thm}\label{uperbound}
 Assume that $\textbf{f}_S\in\textbf{L}^2(\Omega_S)$ and $\textbf{f}_D\in\textbf{L}^2(\Omega_D)$ satisfy the 
conditions of Theorems \ref{ue} and \ref{uh}. 
Let $(\textbf{u},(p,\lambda))\in \textbf{H}\times \textbf{Q}$ be the exact solution and
 $(\textbf{u}_h,(p_h,\lambda_h))\in\textbf{H}_h\times \textbf{Q}_h$ be the finite element solution of coupled 
 problem Navier-Stokes/Darcy. Then, 
 the a posteriori error estimator $\Theta$ satisfies (\ref{Upper}).
 \end{thm}
 
\begin{proof}
 Follows directly from Lemmas \ref{Lemma1} and \ref{Lemma2}.
\end{proof}
\subsection{Efficiency of the a posteriori error estimator}\label{sec:ul}
In order to derive the local lower bounds, we proceed similarly as in \cite{Ca:97} and 
\cite{CD:98} (see also \cite{15}), by applying
inverse inequalities, and the localization technique based on simplex-bubble and face-bubble functions. To this end, we 
  recall some notation and introduce further preliminary results. Given $T\in \mathcal{T}_h$, and 
$E\in \cE(T)$,
we let $b_T$ and $b_E$ be the usual simplexe-bubble and face-bubble 
functions respectively (see (1.5) and (1.6) in \cite{verfurth:96b}). In particular, $b_T$ satisfies 
$b_T\in \mathbb{P}^3(T)$, $supp(b_T)\subseteq T$, $b_T=0 \mbox{ sur } \partial T$, and $0\leq b_T\leq 1\mbox{ on } T $.
Similarly, $b_E\in \mathbb{P}^2(T)$, $supp(b_E)\subseteq 
\omega_E:=\left\{T'\in \mathcal{T}_h:  E\in\cE (T')\right\}$, 
$b_E=0\mbox{  on  } \partial T\smallsetminus E$ and $0\leq b_E\leq 1\mbox{ in } \omega_E$.
We also recall from \cite{verfurth:94a} that, given $k\in\mathbb{N}$, there exists an extension operator
 $L: C(E)\longrightarrow C(T)$ that satisfies $L(p)\in \mathbb{P}^k(T)$ and $L(p)_{|E}=p, \forall p\in \mathbb{P}^k(E)$.
A corresponding vectorial version of $L$, that is, the componentwise application of $L$, is denoted by 
$\textbf{L}$. Additional properties of $b_T$, $b_E$ and $L$ are collected in the following lemma (see \cite{verfurth:94a})

\begin{lem}
 Given $k\in \mathbb{N}^*$, there exist positive constants depending only on $k$ and shape-regularity of the triangulations 
(minimum angle condition), such that for each simplexe $T$ and $E\in \cE(T)$ there hold
\begin{eqnarray}\label{cl1}
\parallel q \parallel_{0,T}&\lesssim&\parallel qb_T^{1/2}\parallel_{0,T}\lesssim
 \parallel q\parallel_{0,T}, \forall q\in \mathbb{P}^k(T)\\\label{cl2}
\|\nabla(q b_T)\|_{0,T}&\lesssim&  h_T^{-1}\parallel q \parallel_{0,T}, \forall q\in \mathbb{P}^k(T)\\\label{cl3}
\parallel p\parallel_{0,E}&\lesssim&\parallel b_E^{1/2}p\parallel_{0,E}\lesssim \parallel p\parallel_{0,E},
\forall p\in \mathbb{P}^k(E)\\\label{cl4}
\parallel L(p)\parallel_{0,T} +h_E\|\nabla(L(p))\|_{0,T}&\lesssim& h_E^{1/2}\parallel p\parallel_{0,E}
\forall p\in \mathbb{P}^k(E)
\end{eqnarray}
\end{lem}

To prove local efficiency for $\omega\subset \Omega:=\Omega_S\cup\Sigma\cup\Omega_D$, let us denote by 
\begin{eqnarray*}
 \|(\textbf{v},(q,\xi))\|_{h,w}^2&:=&
 \displaystyle\sum_{E\in\cE_h(\bar{\omega}\cap \bar{\Omega}_S)}h_E^{-1}
 \left(\|\textbf{v}\|_{1,\omega_E}^2+\|q_S\|_{\omega_E}^2\right)+\|\textbf{v}_D\|_{\textbf{H}(\dive;
 \omega\cap\Omega_D)}^2\\
 &+&\|q_D\|_{L^2(\omega\cap\Omega_D)}^2+\|\xi\|_{1/2,\Sigma\cap\bar{\omega}}^2,
\end{eqnarray*}
where
\begin{eqnarray}
 \omega_E&:=&\cup\left\{T'\in\cT_h^S:\mbox{   } E\in\cE(T)\right\}.
\end{eqnarray}
Recall further the notation for the velocity error $\textbf{e}_{\textbf{u}} = \textbf{u}-\textbf{u}_h$, the
pressure error $e_p=p-p_h$ and the Lagrange multiplier $e_{\lambda}=\lambda-\lambda_h$.
The main result of this subsection can be stated as follows:
\begin{thm} \label{Lowerbound}
Assume that $\textbf{f}_S\in\textbf{L}^2(\Omega_S)$ and $\textbf{f}_D\in\textbf{L}^2(\Omega_D)$ satisfy the 
conditions of Theorem \ref{ue} and \ref{uh}. 
Let $(\textbf{u},(p,\lambda))\in \textbf{H}\times \textbf{Q}$ be the exact solution and
 $(\textbf{u}_h,(p_h,\lambda_h))\in\textbf{H}_h\times \textbf{Q}_h$ be the finite element solution of coupled 
 problem Navier-Stokes/Darcy. Then, the local error estimator $\Theta_T$ satisfies:
 \begin{eqnarray}\label{leb}
  \Theta_T&\lesssim&\|(\textbf{e}_{\textbf{u}},(e_p,e_\lambda))\|_{h,\tilde \omega_T}
  +\displaystyle\sum_{T'\subset \tilde \omega_T}\zeta_{T'}, \forall T\in \cT_h,
 \end{eqnarray}
 where $\tilde \omega_T$ is a finite union of neighbording elements of $T$.

\end{thm}
\begin{proof}
To establish the lower error bound (\ref{leb}), we will make extensive use of the original system of 
equations given by (\ref{I1}) to (\ref{cd3}), which is recovered from the mixed formulation (\ref{FV}) by choosing suitable 
test functions and integrating by parts backwardly the corresponding equations. Thereby, 
 we bound each term of the residual separately.  
 \begin{enumerate}
  \item \textbf{Element residual in $\Omega_S$.} Set 
  $\textbf{w}_T:=\textbf{r}_{S,T}b_T\in [H_0^1(T)]^2$ and consider 
  \begin{eqnarray*}
  (\textbf{r}_{S,T},\textbf{w}_T)_T=\\
  \int_T\left(\textbf{f}_S+2\mu\dive(\textbf{e}(\textbf{u}_{h,S}))-
 \nabla p_{h,S}-\rho\left(\textbf{u}_{h,S}\cdot\nabla\right)\textbf{u}_{h,S}-\frac{\rho}{2}
 \textbf{u}_{h,S}\dive\textbf{u}_{h,S}\right)\cdot\textbf{w}_T.
  \end{eqnarray*}
  Introduce $\textbf{f}_S$ and 
use the formulation (\ref{FV}) to get, 
\begin{eqnarray*}
 \int_T\textbf{r}_{S,T}\cdot\textbf{w}_T&=&\int_T(\textbf{f}_{S}-\textbf{f}_{S,h})\cdot\textbf{w}_T\\
 &+&\int_T(2\mu \textbf{e}(\textbf{u}_{S}):\nabla\textbf{w}_T)+
 \rho\int_T[(\textbf{u}_{S}\cdot\nabla)\textbf{u}_{S}]\cdot\textbf{w}_T-p_S\dive\textbf{w}_T\\
 &+&\int_T\left[2\mu\dive(\textbf{e}(\textbf{u}_{h,S}))-
 \nabla p_{h,S}-\rho\left(\textbf{u}_{h,S}\cdot\nabla\right)\textbf{u}_{h,S}\right]\cdot\textbf{w}_T\\
 &-&\int_T\left[\frac{\rho}{2}
 \textbf{u}_{h,S}\dive\textbf{u}_{h,S}\right]\cdot\textbf{w}_T.
\end{eqnarray*}
Integrating by parts we get,
\begin{eqnarray*}
 \int_T\textbf{r}_{S,T}\cdot\textbf{w}_T
 &=&
 \int_T(\textbf{f}_S-\textbf{f}_{S,h})\cdot\textbf{w}_T+
 2\mu\int_{T}\textbf{e}(\textbf{e}_{\textbf{u}_S}):\nabla(\textbf{w}_T)-\int_Te_p\dive\textbf{w}_T\\
 &+&
 \int_T\left[\rho(\textbf{u}_{S}\cdot\nabla)\textbf{u}_{S}-\frac{\rho}{2}
 \textbf{u}_{h,S}\dive\textbf{u}_{h,S}-\rho\left(\textbf{u}_{h,S}\cdot\nabla\right)\textbf{u}_{h,S}
 \right]\cdot\textbf{w}_T
\end{eqnarray*}
Cauchy-Schwarz inequality implies that, 
\begin{eqnarray*}
 \int_T\textbf{r}_{S,T}\cdot\textbf{w}_T&\lesssim&
 \|\textbf{f}_S-\textbf{f}_{S,h}\|_{0,T}\|\textbf{w}_T\|_{0,T}+(2\mu\|\textbf{e}_{\textbf{u}_S}\|_{1,T}
 +\|e_{p_S}\|_{0,T})\|\nabla\textbf{w}_T\|_{0,T}\\
 &+&\left|\int_T\left[\rho(\textbf{u}_{S}\cdot\nabla)\textbf{u}_{S}-\frac{\rho}{2}
 \textbf{u}_{h,S}\dive\textbf{u}_{h,S}-\rho\left(\textbf{u}_{h,S}\cdot\nabla\right)\textbf{u}_{h,S}
 \right]\cdot\textbf{w}_T\right|
\end{eqnarray*}
The inverse inequalities (\ref{cl1}), (\ref{cl2}), the obvious relation 
$\parallel \textbf{w}_T\parallel_{0,T}\leqslant \parallel \textbf{r}_{s,T}\parallel_{0,T}$ and 
the fact that $\|\textbf{u}_S\|_{1,\Omega_S}$ (see \cite[Lemma 5 and Lemma 6]{MR:2016}) 
and $\|\textbf{u}_{S,h}\|_{1,\Omega_S}$ (see \cite[Lemma 12]{MR:2016}) are both bounded lead to,
\begin{eqnarray*}
 \|\textbf{r}_{S,T}\|_{0,T}&\lesssim& \|\textbf{f}_S-\textbf{f}_{S,h}\|_{0,T}+
 h_T^{-1}\|\nabla\textbf{e}_{\textbf{u}_S}\|_{0,T}+h_T^{-1}\|e_{p_S}\|_{0,T}.
\end{eqnarray*}
As $h_{T}^{-1}\leq h_E^{-1}, \forall E\in\cE(T)$, then we deduce,
\begin{eqnarray}\label{LW1}
 \|\textbf{r}_{S,T}\|_{0,T}\lesssim \|(\textbf{e}_{\textbf{u}},(e_p,e_{\lambda}))\|_{h,w_T}+\zeta_S.
\end{eqnarray}

\item \textbf{Element residual in $\Omega_D$.} Set $\textbf{w}_T:=\textbf{r}_{D,T}b_T\in [H_0^1(T)]^2$, we use (\ref{FV}) and 
Integrate by parts to obtain: 
\begin{eqnarray*}
 \int_T\textbf{r}_{D,T}\cdot\textbf{w}_T&=&\int_T(\textbf{f}_{h,D}-\textbf{K}^{-1}\textbf{u}_{h,D}-
 \nabla p_{h,D})\cdot\textbf{w}_T\\
 &=&\int_T(\textbf{f}_{h,D}-\textbf{K}^{-1}\textbf{u}_{h,D}-\nabla p_{h,D})\cdot\textbf{w}_T\\
 &+&\int_T(\textbf{K}^{-1}\textbf{u}_D-\textbf{f}_D)\cdot\textbf{w}_T-p_D\dive\textbf{w}_{T}\\
 &=&\int_T-(\textbf{f}_D-\textbf{f}_{h,D})\cdot\textbf{w}_T+\int_T (\textbf{K}^{-1}\textbf{e}_{\textbf{u}_D}
 \cdot\textbf{w}_T+e_{p_D}\dive\textbf{w}_T).
\end{eqnarray*}
As before Cauchy-Schwarz inequality and the inverse inequalities  (\ref{cl1})-(\ref{cl2})   lead to,
\beq\nonumber
 h_T\|\textbf{r}_{D,T}\|_{0,T}\lesssim h_T\parallel \textbf{f}_D-\textbf{f}_{h,D}\parallel_{0,T}+
 \parallel \textbf{K}^{-1}\textbf{e}_{\textbf{u}_D}\parallel_{0,T}+\parallel e_{p_D}\parallel_{0,T}.
\eeq
 Thereby,
 \begin{eqnarray}\label{LW2}
  h_T\|\textbf{r}_{D,T}\|_{0,T}\lesssim\|(\textbf{e}_{\textbf{u}},(e_p,e_{\lambda}))\|_{h,w_T}+\zeta_D.
 \end{eqnarray}

\item \textbf{Curl element residual in } $\Omega_D$. \\
For $T\in\cT_{h}^D$, we set $C_T=\curl(\textbf{f}_{h,D}-\textbf{K}^{-1}\textbf{u}_{h,D})$ and 
$w_T=C_Tb_T$. Hence we notice that $\curl (w_T)$ belongs to $\textbf{H}$ and is divergence free, therefore
by (\ref{FV}) we have 
\begin{eqnarray*}
 \textbf{a}(\textbf{u}_D,\curl(w_T))=(\textbf{f}_D,\curl(w_T))_D,
\end{eqnarray*}
or equivalently
\begin{eqnarray}\label{curl}
 \int_T(\textbf{K}^{-1}\textbf{u}_{D}-\textbf{f}_D)\cdot\curl (w_T)=0.
\end{eqnarray}
But by Green's formula we may write 
\begin{eqnarray*}
 \int_TC_T w_T=\int_T\curl(\textbf{f}_{h,D}-\textbf{f}_D)w_T+
 \int_T(\textbf{f}_D-\textbf{K}^{-1}\textbf{u}_{h,D})\cdot\curl (w_T),
\end{eqnarray*}
and by using (\ref{curl}) we deduce that 
\[
\int_T C_T w_T=\int_T \curl  (\textbf{f}_{h,D}-\textbf{f}_D)w_T+
\int_T [\textbf{K}^{-1}(\bu_D-\textbf{u}_{h,D})]\cdot \curl (w_T).
\]
By Cauchy-Schwarz inequality we obtain
 \[
\int_T C_T w_T\leq \|\curl  (\textbf{f}_{h,D}-\textbf{f}_D)\|_{0,T} \|w_T\|_{0,T}+
\parallel \textbf{K}^{-1}\textbf{e}_{\textbf{u}_D}\parallel_{0,T} \|\curl w_T\|_{0,T}.
\]
Again the inverse inequalities  (\ref{cl1})-(\ref{cl2}) allows to get
\beq\nonumber
h_T\|\curl(\textbf{f}_{h,D}-\textbf{K}^{-1}\textbf{u}_{h,D})\|_{0,T}\lesssim 
\parallel \textbf{K}^{-1}\textbf{e}_{\textbf{u}_D}\parallel_{0,T}+ h_T \|\curl 
(\textbf{f}_{h,D}- \textbf{f}_D)\|_{0,T},
\eeq
let
\begin{eqnarray}\label{LW3}
 h_T\|\curl(\textbf{f}_{h,D}-\textbf{K}^{-1}\textbf{u}_{h,D})\|_{0,T}&\lesssim &
 \|(\textbf{e}_{\textbf{u}},(e_p,e_{\lambda}))\|_{h,w_T}+\zeta_D.
\end{eqnarray}

\item \textbf{Divergence element in $\Omega_*$, $*\in\{S,D\}$.} We directly see that 
$$\dive (\textbf{u}_{*}-\textbf{u}_{h,*})=-\dive\textbf{u}_{h,*},  \forall *\in\{S,D\}, $$
hence by Cauchy-Schwarz inequality we conclude 
\begin{eqnarray}\label{LW4}
 \|\dive \textbf{u}_{h,*}\|_{0,T}&\lesssim& \|\dive \textbf{e}_{\textbf{u}_*}\|_{0,T}, *\in\{S,D\}.
\end{eqnarray}
\item \textbf{Normal jump in $\Omega_S$.} 
For each edge $E\in\cE_h(\Omega_S)$, we consider $w_E=T_1\cup T_2$. As 
$\textbf{J}_{E,\textbf{n}_E}\in[\mathbb{P}^0(E)]^2$ we set 
\begin{eqnarray*}
 \textbf{w}_E:=-\textbf{J}_{E,\textbf{n}_E}b_E\in [H_0^1(w_E)]^2.
\end{eqnarray*}
First the weak formulation (\ref{FV}) yieds 
\begin{eqnarray*}
 \textbf{a}(\textbf{u},\textbf{w}_E)+\textbf{b}(\textbf{w}_E,p)=(\textbf{f},\textbf{w}_E)_{w_E},
\end{eqnarray*}
that is equivalent to 
\begin{eqnarray}\nonumber
 \int_{w_E}\textbf{f}_S\cdot\textbf{w}_E&=&
 \int_{w_E}[2\mu\textbf{e}(\textbf{u}_S)-p_S\textbf{I}]:\textbf{e}(\textbf{w}_E)\\\nonumber\label{jump}
 &+&
 \int_{w_E}\left[\rho(\textbf{u}_{S}\cdot\nabla)\textbf{u}_{S}
 \right]\cdot\textbf{w}_E\\
 &+&\int_{\partial \omega_E}[p_S\textbf{I}-
 2\mu\textbf{e}(\textbf{u}_{S})] \textbf{n}_E\cdot\textbf{w}_E
\end{eqnarray}
By elementwise partial integration we further have 
\begin{eqnarray*}
 -\int_E\textbf{J}_{E,\textbf{n}_E}\cdot\textbf{w}_E&=&
\int_{\omega_E}(2\mu \textbf{e}(\textbf{u}_{h,S})-p_{h,S}\textbf{I}):\textbf{e}( \textbf{w}_E)\\
 &-&\sum_{i=1}^2\int_{T_i}(-2\mu\dive \textbf{e}(\textbf{u}_{h,S})+\nabla p_{h,S})\cdot\textbf{w}_E.
\end{eqnarray*}
Hence, by previous identity (\ref{jump}) we get 
\begin{eqnarray*}
 -\int_{E}\textbf{J}_{E,\textbf{n}_E}\cdot\textbf{w}_E&=& \displaystyle
 \sum_{i=1}^2\int_{T_i}[\textbf{f}_S-(-2\mu\dive\textbf{e}(\textbf{u}_{h,S})+\nabla p_{h,S})]\cdot\textbf{w}_E\\
 &-&\int_{w_E}[2\mu\textbf{e}(\textbf{e}_{\textbf{u}_{h,S}})-e_{p_S}\textbf{I}]:\textbf{e}(\textbf{w}_E)\\
 &+&\int_{w_E}\left[\rho(\textbf{u}_{S}\cdot\nabla)\textbf{u}_{S}\right]\cdot\textbf{w}_E
\end{eqnarray*}
We introduce the approximation $\textbf{f}_{S,h}$ of $\textbf{f}_S$, use the Cauchy-Schwarz inequality, the inverse 
inequalities (\ref{cl3})-(\ref{cl4}) and
the fact that $\|\textbf{u}_S\|_{1,\Omega_S}$ (see \cite[Lemma 5 and Lemma 6]{MR:2016}) 
 is bounded,
to get,
\begin{eqnarray*}
 \parallel\textbf{J}_{E,\textbf{n}_E}\parallel_{0,E}&\lesssim &
  h_E^{1/2}\left(\sum_{i=1}^2
( \parallel\textbf{f}_S-\textbf{f}_{h,S}\parallel_{0,T_i}+
 \parallel \textbf{r}_{S,T_i}\parallel_{0,T_i})\right)\\
 &+&  h_E^{-1/2}\left(\|\nabla(\textbf{e}_{\textbf{u}_S})\|_{0,w_E}+
 \parallel e_{p_S}\parallel_{0,w_E}\right) 
\end{eqnarray*}
As $h_E\leq 1$, then by (\ref{LW1}) we obtain,
\begin{eqnarray}\label{LW5}
 \parallel\textbf{J}_{E,\textbf{n}_E}\parallel_{0,E}&\lesssim &
 \|(\textbf{e}_{\textbf{u}},(e_p,e_{\lambda}))\|_{h,w_E}+\zeta_S.
\end{eqnarray}

\item \textbf{Interface elements on $\Sigma$.} To estimate the interface elements, we fix an edge $E$ in 
$\Sigma$ and for a constant $r_E$ fixed later on and a unit vector $\textbf{N}$, we consider
$\textbf{w}_E=(\textbf{w}_{E,S},\textbf{w}_{E,D})$ such that:
\begin{eqnarray}
 \textbf{w}_E=r_Eb_E\textbf{N}.
\end{eqnarray}
The vector $\textbf{w}_E$
 clearly, belongs to $\textbf{H}$. Hence the weak formulation (\ref{FV}) yieds 
\[
\textbf{a}(\bu, \bw_E)+\textbf{b}(\bw_E, p)=(\bff, \bw_E)_{\omega_E},
\]
that is equivalent to
\bea\label{serge12:03:4}
&&\int_{T_S} (2\mu \textbf{e}(\bu_S): \textbf{e}(\bw_E)-p_S\dive \bw_E)+
\int_{T_S} ( \textbf{K}^{-1}\bu_D\cdot \bw_E-p_D\dive \bw_E)
\\
&&
\nonumber
+\frac{\mu\alpha_d}{\sqrt{\tau\cdot\kappa\cdot\tau}} (\textbf{u}_S\cdot\tau,
 \textbf{w}_{E,S}\cdot\tau)_E+
 \int_{w_E}\left[\rho(\textbf{u}_{S}\cdot\nabla)\textbf{u}_{S}\right]\cdot\textbf{w}_E
=(\bff, \bw_E)_{\omega_E},
\eea
where $T_S$ (resp. $T_D$) is the unique triangle included in $\bar \Omega_S$ 
(resp. $\bar \Omega_D$) having $E$ as edge.\\
On the other hand, integrating by parts in $T_S$ and $T_D$ yieds
\beqs
&&\int_{T_S} (2\mu \textbf{e}(\bu_{h,S}): \textbf{e}(\bw_{E,S})-p_{h,S}\dive \bw_{E,S})+
\int_{T_D} ( \textbf{K}^{-1}\bu_{h,D}\cdot \bw_{E,D}-p_{h,D}\dive \bw_{E,D})
\\
&&
\nonumber
+
 \frac{\mu\alpha_d}{\sqrt{\tau\cdot\kappa\cdot\tau}} (\textbf{u}_{h,S}\cdot\tau,
 \textbf{w}_{E,S}\cdot\tau)_E
\\
&&
=
-\int_{T_S} (2\mu \dive \textbf{e}(\bu_{h,S})-\nabla p_{h,S})\cdot \bw_{E,S}+
\int_{T_D} ( \textbf{K}^{-1}\bu_{h,D}\cdot \bw_{E,D}+\nabla p_{h,D})\cdot\bw_{E,D}
\\
&&
\nonumber
+\frac{\mu\alpha_d}{\sqrt{\tau\cdot\kappa\cdot\tau}} (\textbf{u}_{h,S}\cdot\tau,
 \textbf{w}_{E,S}\cdot\tau)_E
\\
&&
-\int_E \left([p_{h,S}]_E \bw_{E,S}\cdot \bn_E-2\mu [\textbf{e}(\bu_{h,S}) \bn_E]\cdot  \bw_{E,S}\right).
\eeqs
Subtracting this identity to (\ref{serge12:03:4}) we find
\beqs
&&\int_E ([p_h]_E \bw_E\cdot \bn_E-2\mu (\textbf{e}(\bu_{h,S}) \bn_E\cdot  \bw_{E,S})
-\frac{\mu\alpha_d}{\sqrt{\tau\cdot\kappa\cdot\tau}} (\textbf{u}_{h,S}\cdot\tau,
 \textbf{w}_{E,S}\cdot\tau)_E
\\
\nonumber
&&=\int_{T_S} (2\mu \textbf{e}(\textbf{e}_{\textbf{u}_S}): \textbf{e}(\bw_{E,S})-e_{p_S}\dive \bw_{E,S})+
\int_{T_D} (  \textbf{K}^{-1}\textbf{e}_{\textbf{u}_D}\cdot \bw_{E,D}-e_{p_D}\dive \bw_{E,D})
\\
&&
\nonumber
+\frac{\mu\alpha_d}{\sqrt{\tau\cdot\kappa\cdot\tau}} (\textbf{e}_{\textbf{u}_S}\cdot\tau,
 \textbf{w}_{E,s}\cdot\tau)_E
-\int_{T_S}\left[\rho(\textbf{u}_{S}\cdot\nabla)\textbf{u}_{S}\right]\cdot\textbf{w}_E\\
&&
-\int_{T_S} (\bff_S+2\mu \dive \textbf{e}(\bu_{h,S})-\nabla p_{h,S})\cdot \bw_{E,S}-
\int_{T_D} (\bff_D- \textbf{K}^{-1}\bu_{h,D}-\nabla p_{h,D})\cdot\bw_{E,D}.
\eeqs
In that last terms introducing the element residual $\textbf{r}_{*,T}$, $*\in\{S,D\}$, we arrive at
\bea\label{Interfacejump}\nonumber
&&\int_E ([p_{h,S}]_E \bw_E\cdot \bn_E-2\mu (\textbf{e}(\bu_{h,S}) \bn_E\cdot  \bw_{E,S})
-\frac{\mu\alpha_d}{\sqrt{\tau\cdot\kappa\cdot\tau}} (\textbf{u}_{h,S}\cdot\tau,
 \textbf{w}_{E,S}\cdot\tau)_E
\\\nonumber
&&=\int_{T_S} (2\mu \textbf{e}(\textbf{e}_{\textbf{u}_S}): \textbf{e}(\bw_{E,S})-e_{p_S}\dive \bw_{E,S})+
\int_{T_D} (  \textbf{K}^{-1}\textbf{e}_{\textbf{u}_D}\cdot \bw_{E,D}-e_{p_D}\dive \bw_{E,D})
\\
&&
+\frac{\mu\alpha_d}{\sqrt{\tau\cdot\kappa\cdot\tau}} (\textbf{e}_{\textbf{u}_S}\cdot\tau,
 \textbf{w}_{E,S}\cdot\tau)_E
-\int_{T_S}\left[\rho(\textbf{u}_{S}\cdot\nabla)\textbf{u}_{S}\right]\cdot\textbf{w}_E\\\nonumber
&&
-\int_{T_S} (\textbf{f}_S-\textbf{f}_{h,S}+\textbf{r}_{S,T})\cdot \bw_{E,S}-
\int_{T_D} (\textbf{f}_D-\textbf{f}_{h,D}+\textbf{r}_{D,T})\cdot\bw_{E,D}.
\eea
\begin{enumerate}
 \item To estimate the term $\displaystyle\sum_{E\in\cE(T)\cap\cE_h(\bar{\Sigma})}
   \left\|\frac{\alpha_d\mu}{\sqrt{\tau\cdot\kappa\cdot\tau}}\textbf{u}_{h,S}\cdot\tau+2\mu\textbf{n}_S\cdot
 \textbf{e}(\textbf{u}_{h,S})\cdot\tau\right\|_{0,E}^2$, 
 we take $$r_E=\frac{\alpha_d\mu}{\sqrt{\tau\cdot\kappa\cdot\tau}}\textbf{u}_{h,S}\cdot\tau+2\mu\textbf{n}_S\cdot
 \textbf{e}(\textbf{u}_{h,S})\cdot\tau \mbox{  and  }
 \textbf{N}=\tau. $$
 With this choice, $\textbf{w}_E\cdot\textbf{n}_S+\textbf{w}_E\cdot\textbf{n}_D=0 \mbox{  on  } \Sigma$. And thus,
 the identity (\ref{Interfacejump}) and the inverse inequality (\ref{cl3}) yield,
 \beqs
\|r_E\|_E^2&\lesssim &
\int_{T_S} (2\mu \textbf{e}(\textbf{e}_{\textbf{u}_S}): \textbf{e}(\bw_{E,S})-e_{p_S}\dive \bw_{E,S})+
\int_{T_D} ( \textbf{K}^{-1}\textbf{e}_{\textbf{u}_D}\cdot \bw_{E,D}-e_{p_D}\dive \bw_{E,D})
\\
&+&
\nonumber
 \frac{\mu\alpha_d}{\sqrt{\tau\cdot\kappa\cdot\tau}} (\textbf{e}_{\textbf{u}_S}\cdot\tau,
 \textbf{w}_{E,S}\cdot\tau)_E-\int_{T_S}\left[\rho(\textbf{u}_{S}\cdot\nabla)\textbf{u}_{S}\right]\cdot\textbf{w}_{E,S}
\\
&-&
\int_{T_S} (\bff_S-\bff_{h,S}+ \textbf{r}_{S, T})\cdot \bw_{E,S}-
\int_{T_D} (\bff_D-\bff_{h,D}+ \textbf{r}_{D, T})\cdot\bw_{E,D}.
\eeqs
Hence  Cauchy-Schwarz inequality, the inverse inequalities (\ref{cl4}), the upper error bound of 
$\|\textbf{r}_{*,T}\|_{0,T}$ [ i.e. estimates (\ref{LW1}) and (\ref{LW2})], and 
the fact that $\|\textbf{u}_S\|_{1,\Omega_S}$ (see \cite[Lemma 5 and Lemma 6]{MR:2016}) 
 is bounded lead to
 \beq\label{LW6}
\left\|\frac{\alpha_d\mu}{\sqrt{\tau\cdot\kappa\cdot\tau}}\textbf{u}_{h,S}\cdot\tau+2\mu\textbf{n}_S\cdot
 \textbf{e}(\textbf{u}_{h,S})\cdot\tau\right\|_{0,E}\lesssim  \|(\textbf{e}_{\textbf{u}},(e_p,e_{\lambda}))\|_{h,\omega_E}+
 \sum_{T'\subset \omega_E} \zeta_{T'},
\eeq
with $\omega_E= T_S\cup T_D$.
\item To estimate the term $
\displaystyle\sum_{E\in\cE(T)\cap\cE_h(\bar{\Sigma})}
   \|-p_{h,S}+p_{h,D}-2\mu\textbf{n}_S\cdot\textbf{e}(\textbf{u}_{h,S})\cdot\textbf{n}_S\|_{0,E}^2$, 
    we take 
$$r_E=p_{h,D}-p_{h,S}+2\mu \textbf{n}_s\cdot\textbf{e}(\textbf{u}_{h,S})\cdot \textbf{n}_S
\mbox{  and  } \textbf{N}=\bn_S.$$
As before the identity (\ref{Interfacejump}), 
the  inverse inequalities (\ref{cl3}) and (\ref{cl4}), the upper bounds of  
$\|\textbf{r}_{*,T}\|_{0,T}$, $*\in\{S,D\}$ and of $(\ref{LW5})$, and  the fact that $\|\textbf{u}_S\|_{1,\Omega_S}$ 
(see \cite[Lemma 5 and Lemma 6]{MR:2016}) is bounded lead to
\begin{eqnarray}\label{LW7}
  \|p_{h,D}-p_{h,S}+2\mu \textbf{n}_S\cdot\textbf{e}(\textbf{u}_{h,S})\cdot 
 \textbf{n}_S\|_{0,E}\lesssim  
 \|(\textbf{e}_{\textbf{u}},(e_p,e_{\lambda}))\|_{h,\omega_E}+
 \sum_{T'\subset \omega_E} \zeta_{T'}.
\end{eqnarray}
\item For $E\in\cE_h(\bar{\Sigma})$, the term $
\displaystyle\sum_{E\in\cE_h(\bar{\Sigma})}h_E\|p_{h,D}-\lambda_h\|_{0,E}^2$  is bounded as follows:
\begin{eqnarray}\nonumber\label{LW8}
 \displaystyle\sum_{E\in\cE_h(\bar{\Sigma})}h_E\|p_{h,D}-\lambda_h\|_{0,E}^2 
 &\lesssim& \displaystyle\sum_{E\in\cE_h(\bar{\Sigma})}h_E \left(\|\lambda-\lambda_h\|_{0,E}^2+
 \|\lambda-p_{h,D}\|_{0,E}^2\right)\\
 &\lesssim&h\|\lambda-\lambda_h\|_{1/2,\Sigma}^2\lesssim \|(\textbf{e}_{\textbf{u}},(e_p,e_{\lambda}))\|_{h,\omega_E}^2.
\end{eqnarray}
\item Analogously to  \cite[Lemma 4.7]{49}, the term \\
$\displaystyle\sum_{E\in\cE(\bar{\Sigma})}\|\textbf{u}_{h,S}\cdot\textbf{n}_S+\textbf{u}_{h,D}\cdot\textbf{n}_D\|_{0,E}$ 
can be bounded by 
\begin{eqnarray*}
 \|\textbf{u}_{h,S}\cdot\textbf{n}_S+\textbf{u}_{h,D}\cdot\textbf{n}_D\|_{0,E}^2
 &\lesssim& h_{T_S}^{-1}\|\textbf{u}_S-\textbf{u}_{h,S}\|_{0,T_S}^2+h_{T_S}^2\|\nabla (\textbf{u}_S-\textbf{u}_{h,S})\|_{0,T}^2\\\nonumber
 &+&\|\textbf{u}_D-\textbf{u}_{h,D}\|_{0,T_D}^2+h_{T_D}^2\|\dive(\textbf{u}_D-\textbf{u}_{h,D})\|_{0,T}^2,
\end{eqnarray*}
where $E=\partial T_S\cap \partial T_D$.
Now, as for each $*\in\{S,D\}$, $h_{T_*}\leq 1$ and $h_{T_S}^{-1}\leq h_E^{-1}$, then we have the estimate,
\begin{eqnarray}\label{LW9}
 \|\textbf{u}_{h,S}\cdot\textbf{n}_S+\textbf{u}_{h,D}\cdot\textbf{n}_D\|_{0,E}
 &\lesssim&  \|(\textbf{e}_{\textbf{u}},(e_p,e_{\lambda}))\|_{h,\omega_E}, \mbox{  with  } \omega_E=T_S\cup T_D.
\end{eqnarray}

\end{enumerate}
\item \textbf{Tangential jump in $\bar{\Omega}_D$.} Finally, for $E\in\cE_h(\bar{\Omega}_D)$, the terms\\
$\displaystyle\sum_{E\in\cE_h(\Omega_D)}
h_E\|[(\textbf{f}_{h,D}-\textbf{K}^{-1}\textbf{u}_{h,D}-\nabla p_{h,D})\cdot\tau_E]_E\|_{0,E}^2$ \\and 
$\displaystyle\sum_{E\in\cE_h(\partial\Omega_D)}
h_E\|(\textbf{f}_{h,D}-\textbf{K}^{-1}\textbf{u}_{h,D}-\nabla p_{h,D})\cdot\tau_E\|_{0,E}^2$ respectively are bounded 
analogously as 
in 
\cite[Lemma 3.16]{SGR:2016} by: 
\begin{eqnarray}\label{LW10}
 h_E\|[(\textbf{f}_{h,D}-\textbf{K}^{-1}\textbf{u}_{h,D}-\nabla p_{h,D})\cdot\tau_E]_E\|_{0,E}^2
 &\lesssim &\|\textbf{u}_D-\textbf{u}_{h,D}\|_{0,w_E}^2\\\nonumber
 &\lesssim&  \|(\textbf{e}_{\textbf{u}},(e_p,e_{\lambda}))\|_{h,\omega_E}^2,
 \end{eqnarray}
  for all $ E\in \cE_h(\Omega_D)$, where the set $w_E$ is given by  
 $$w_E:=\cup\left\{T'\in\cT_h^D: E\in \cE(T')\right\};$$
 and 
\begin{eqnarray}\label{LW11}
 h_E\|(\textbf{f}_{h,D}-\textbf{K}^{-1}\textbf{u}_{h,D}-\nabla p_{h,D})\cdot\tau_E\|_{0,E}^2
 &\lesssim& \|\textbf{u}_D-\textbf{u}_{h,D}\|_{0,T_E}^2\\\nonumber
 &\lesssim&  \|(\textbf{e}_{\textbf{u}},(e_p,e_{\lambda}))\|_{h,\omega_E}^2,
\end{eqnarray}
for all $E\in\cE_h(\partial \Omega_D)$, with  $T_E$, the triangle of $\cT_h^D$ having $E$ as a edge.

\end{enumerate}
The estimates (\ref{LW1}), (\ref{LW2}), (\ref{LW3}), (\ref{LW4}), (\ref{LW5}), 
(\ref{LW6}), (\ref{LW7}), (\ref{LW8}), (\ref{LW9}), (\ref{LW10}) and (\ref{LW11})
provide the desired local
lower error bound of Theorem \ref{Lowerbound}.
\end{proof}

\section{Summary}\label{sec:r3}
In this paper we have discussed  a posteriori error estimates for a finite element approximation of the 
Navier-Stokes/Darcy system. 
A residual type a posteriori error estimator is provided, that is both reliable and efficient. Many issues 
remain to be addressed in this area, let us mention    other 
types of a posteriori error estimators or 
implementation and convergence analysis of  adaptive finite element methods.
Further it is well known that an internal layer appears at the interface $\Sigma$
as the permeability tensor degenerates, in that case   anisotropic meshes
have to be used in this layer (see for instance \cite{HA:2016,ESK:04}). Hence we intend to   extend   our results to
such anisotropic meshes.
\section{Acknowledgements}
The first author thanks African Institute for Mathematical Sciences (AIMS South
Africa) for hosting him for a two months research visit and Serge Nicaise
(UVHC, FRANCE) for his collaboration.


\end{Large}
\end{document}